\documentclass[12pt, reqno]{amsart}

\usepackage{amsfonts}
\usepackage{amssymb}
\usepackage{latexsym}
\usepackage{graphics}
\usepackage[2emode]{psfrag}
\usepackage{mathrsfs}
\usepackage{amsthm}
\usepackage{amsmath}
\usepackage[all]{xy}
\usepackage{enumerate}
\usepackage[backref=page]{hyperref}
\usepackage{stmaryrd}
\usepackage{fancyhdr}
\usepackage[table]{xcolor}
\usepackage{cite}
\usepackage{color}

\begin{document}

\def\joost{\color{blue}}

\newcommand{\thlabel}[1]{\label{th:#1}}
\newcommand{\thref}[1]{Theorem~\ref{th:#1}}
\newcommand{\selabel}[1]{\label{se:#1}}
\newcommand{\seref}[1]{Section~\ref{se:#1}}
\newcommand{\lelabel}[1]{\label{le:#1}}
\newcommand{\leref}[1]{Lemma~\ref{le:#1}}
\newcommand{\prlabel}[1]{\label{pr:#1}}
\newcommand{\prref}[1]{Proposition~\ref{pr:#1}}
\newcommand{\colabel}[1]{\label{co:#1}}
\newcommand{\coref}[1]{Corollary~\ref{co:#1}}
\newcommand{\relabel}[1]{\label{re:#1}}
\newcommand{\reref}[1]{Remark~\ref{re:#1}}
\newcommand{\exlabel}[1]{\label{ex:#1}}
\newcommand{\exref}[1]{Example~\ref{ex:#1}}
\newcommand{\delabel}[1]{\label{de:#1}}
\newcommand{\deref}[1]{Definition~\ref{de:#1}}
\newcommand{\eqlabel}[1]{\label{eq:#1}}
\newcommand{\equref}[1]{(\ref{eq:#1})}

\newcommand{\Hom}{{\sf Hom}}
\newcommand{\Alg}{{\sf Alg}}
\newcommand{\Meas}{{\sf Meas}}
\newcommand{\Comeas}{{\sf Comeas}}
\newcommand{\End}{{\sf End}}
\newcommand{\Ext}{{\sf Ext}}
\newcommand{\Fun}{{\sf Fun}}
\newcommand{\Mor}{{\sf Mor}}
\newcommand{\Aut}{{\sf Aut}}
\newcommand{\Hopf}{{\sf Hopf}}
\newcommand{\sHopf}{{\sf sHopf}}
\newcommand{\opHopf}{{\sf opHopf}}
\newcommand{\opsHopf}{{\sf opsHopf}}
\newcommand{\Coalg}{{\sf Coalg}}
\newcommand{\Ann}{{\sf Ann}}
\newcommand{\Ker}{{\sf Ker}}
\renewcommand{\ker}{{\sf Ker}}
\newcommand{\Coker}{{\sf Coker}}
\newcommand{\im}{{\sf Im}}
\newcommand{\coim}{{\sf Coim}}
\newcommand{\Trace}{{\sf Trace}}
\newcommand{\Char}{{\sf Char}}
\newcommand{\Mod}{{\sf Mod}}
\newcommand{\Vect}{{\sf Vect}}
\newcommand{\Spec}{{\sf Spec}}
\newcommand{\Span}{{\sf Span}}
\newcommand{\sgn}{{\sf sgn}}
\newcommand{\Id}{{\sf Id}}
\newcommand{\Com}{{\sf Com}}
\newcommand{\codim}{{\sf codim}}
\newcommand{\Mat}{{\sf Mat}}
\newcommand{\Coint}{{\rm Coint}}
\newcommand{\Incoint}{{\sf Incoint}}
\newcommand{\can}{{\sf can}}
\newcommand{\Bim}{{\sf Bim}}
\newcommand{\CAT}{{\sf CAT}}
\newcommand{\sign}{{\sf sign}}
\newcommand{\kar}{{\sf kar}}
\newcommand{\rad}{{\sf rad}}
\newcommand{\Rat}{{\sf Rat}}
\newcommand{\Cob}{{\sf Cob}}
\newcommand{\ev}{{\sf ev}}
\newcommand{\sd}{{\sf d}}
\def\colim{{\sf colim}\,}
\def\tildej{\tilde{\jmath}}
\def\barj{\bar{\jmath}}

\def\Ab{\underline{\underline{\sf Ab}}}
\def\lan{\langle}
\def\ran{\rangle}
\def\ot{\otimes}
\def\bul{\bullet}
\def\ubul{\underline{\bullet}}

\def\id{\textrm{{\small 1}\normalsize\!\!1}}
\def\To{{\multimap\!\to}}
\def\bigperp{{\LARGE\textrm{$\perp$}}} 
\newcommand{\QED}{\hspace{\stretch{1}}
\makebox[0mm][r]{$\Box$}\\}

\def\AA{{\mathbb A}}
\def\BB{{\mathbb B}}
\def\CC{{\mathbb C}}
\def\DD{{\mathbb D}}
\def\EE{{\mathbb E}}
\def\FF{{\mathbb F}}
\def\GG{{\mathbb G}}
\def\HH{{\mathbb H}}
\def\II{{\mathbb I}}
\def\JJ{{\mathbb J}}
\def\KK{{\mathbb K}}
\def\LL{{\mathbb L}}
\def\MM{{\mathbb M}}
\def\NN{{\mathbb N}}
\def\OO{{\mathbb O}}
\def\PP{{\mathbb P}}
\def\QQ{{\mathbb Q}}
\def\RR{{\mathbb R}}
\def\TT{{\mathbb T}}
\def\UU{{\mathbb U}}
\def\VV{{\mathbb V}}
\def\WW{{\mathbb W}}
\def\XX{{\mathbb X}}
\def\YY{{\mathbb Y}}
\def\ZZ{{\mathbb Z}}

\def\aa{{\mathfrak A}}
\def\bb{{\mathfrak B}}
\def\cc{{\mathfrak C}}
\def\dd{{\mathfrak D}}
\def\ee{{\mathfrak E}}
\def\ff{{\mathfrak F}}
\def\gg{{\mathfrak G}}
\def\hh{{\mathfrak H}}
\def\ii{{\mathfrak I}}
\def\jj{{\mathfrak J}}
\def\kk{{\mathfrak K}}
\def\ll{{\mathfrak L}}
\def\mm{{\mathfrak M}}
\def\nn{{\mathfrak N}}
\def\oo{{\mathfrak O}}
\def\pp{{\mathfrak P}}
\def\qq{{\mathfrak Q}}
\def\rr{{\mathfrak R}}
\def\tt{{\mathfrak T}}
\def\uu{{\mathfrak U}}
\def\vv{{\mathfrak V}}
\def\ww{{\mathfrak W}}
\def\xx{{\mathfrak X}}
\def\yy{{\mathfrak Y}}
\def\zz{{\mathfrak Z}}

\def\aaa{{\mathfrak a}}
\def\bbb{{\mathfrak b}}
\def\ccc{{\mathfrak c}}
\def\ddd{{\mathfrak d}}
\def\eee{{\mathfrak e}}
\def\fff{{\mathfrak f}}
\def\ggg{{\mathfrak g}}
\def\hhh{{\mathfrak h}}
\def\iii{{\mathfrak i}}
\def\jjj{{\mathfrak j}}
\def\kkk{{\mathfrak k}}
\def\lll{{\mathfrak l}}
\def\mmm{{\mathfrak m}}
\def\nnn{{\mathfrak n}}
\def\ooo{{\mathfrak o}}
\def\ppp{{\mathfrak p}}
\def\qqq{{\mathfrak q}}
\def\rrr{{\mathfrak r}}
\def\sss{{\mathfrak s}}
\def\ttt{{\mathfrak t}}
\def\uuu{{\mathfrak u}}
\def\vvv{{\mathfrak v}}
\def\www{{\mathfrak w}}
\def\xxx{{\mathfrak x}}
\def\yyy{{\mathfrak y}}
\def\zzz{{\mathfrak z}}

\newcommand{\aA}{\mathscr{A}}
\newcommand{\bB}{\mathscr{B}}
\newcommand{\cC}{\mathscr{C}}
\newcommand{\dD}{\mathscr{D}}
\newcommand{\eE}{\mathscr{E}}
\newcommand{\fF}{\mathscr{F}}
\newcommand{\gG}{\mathscr{G}}
\newcommand{\hH}{\mathscr{H}}
\newcommand{\iI}{\mathscr{I}}
\newcommand{\jJ}{\mathscr{J}}
\newcommand{\kK}{\mathscr{K}}
\newcommand{\lL}{\mathscr{L}}
\newcommand{\mM}{\mathscr{M}}
\newcommand{\nN}{\mathscr{N}}
\newcommand{\oO}{\mathscr{O}}
\newcommand{\pP}{\mathscr{P}}
\newcommand{\qQ}{\mathscr{Q}}
\newcommand{\rR}{\mathscr{R}}
\newcommand{\sS}{\mathscr{S}}
\newcommand{\tT}{\mathscr{T}}
\newcommand{\uU}{\mathscr{U}}
\newcommand{\vV}{\mathscr{V}}
\newcommand{\wW}{\mathscr{W}}
\newcommand{\xX}{\mathscr{X}}
\newcommand{\yY}{\mathscr{Y}}
\newcommand{\zZ}{\mathscr{Z}}

\newcommand{\Aa}{\mathcal{A}}
\newcommand{\Bb}{\mathcal{B}}
\newcommand{\Cc}{\mathcal{C}}
\newcommand{\Dd}{\mathcal{D}}
\newcommand{\Ee}{\mathcal{E}}
\newcommand{\Ff}{\mathcal{F}}
\newcommand{\Gg}{\mathcal{G}}
\newcommand{\Hh}{\mathcal{H}}
\newcommand{\Ii}{\mathcal{I}}
\newcommand{\Jj}{\mathcal{J}}
\newcommand{\Kk}{\mathcal{K}}
\newcommand{\Ll}{\mathcal{L}}
\newcommand{\Mm}{\mathcal{M}}
\newcommand{\Nn}{\mathcal{N}}
\newcommand{\Oo}{\mathcal{O}}
\newcommand{\Pp}{\mathcal{P}}
\newcommand{\Qq}{\mathcal{Q}}
\newcommand{\Rr}{\mathcal{R}}
\newcommand{\Ss}{\mathcal{S}}
\newcommand{\Tt}{\mathcal{T}}
\newcommand{\Uu}{\mathcal{U}}
\newcommand{\Vv}{\mathcal{V}}
\newcommand{\Ww}{\mathcal{W}}
\newcommand{\Xx}{\mathcal{X}}
\newcommand{\Yy}{\mathcal{Y}}
\newcommand{\Zz}{\mathcal{Z}}

\def\units{{\mathbb G}_m}
\def\rightact{\hbox{$\leftharpoonup$}}
\def\leftact{\hbox{$\rightharpoonup$}}

\def\*C{{}^*\hspace*{-1pt}{\Cc}}
\def\*c{{}^*\hspace*{-1pt}{\cc}}

\def\text#1{{\rm {\rm #1}}}

\def\smashco{\mathrel>\joinrel\mathrel\triangleleft}
\def\cosmash{\mathrel\triangleright\joinrel\mathrel<}

\def\ol{\overline}
\def\ul{\underline}
\def\dul#1{\underline{\underline{#1}}}
\def\Nat{\dul{\rm Nat}}
\def\Set{\dul{\rm Set}}
\def\MCl{\dul{\rm MCl}}

\renewcommand{\subjclassname}{\textup{2000} Mathematics Subject
     Classification}

\newtheorem{proposition}{Proposition}[section] 
\newtheorem{lemma}[proposition]{Lemma}
\newtheorem{corollary}[proposition]{Corollary}
\newtheorem{theorem}[proposition]{Theorem}
\newtheorem{conjecture}[proposition]{Conjecture}

\theoremstyle{definition}
\newtheorem{Definition}[proposition]{Definition}
\newtheorem{definition}[proposition]{Definition}
\newtheorem{example}[proposition]{Example}
\newtheorem{examples}[proposition]{Examples}

\theoremstyle{remark}
\newtheorem{remarks}[proposition]{Remarks}
\newtheorem{remark}[proposition]{Remark}

\title{The Hopf category of Frobenius algebras}

\author[P. Gro{\ss}kopf]{Paul Gro{\ss}kopf}
\address{Paul Gro{\ss}kopf, D\'epartement de Math\'ematiques, Universit\'e Libre de Bruxelles, Belgium}
\email{paul.grosskopf@gmx.at}

\author[J. Vercruysse]{Joost Vercruysse}
\address{Joost Vercruysse, D\'epartement de Math\'ematiques, Universit\'e Libre de Bruxelles, Belgium}
\email{joost.vercruysse@ulb.be}

\begin{abstract}
We show that the universal measuring coalgebras between Frobenius algebras turn the category of Frobenius algebras into a Hopf category (in the sense of \cite{BCV}), and the universal comeasuring algebras between Frobenius algebras turn the category Frobenius algebras into a Hopf opcategory. We also discuss duality and compatibility results between these structures.

Our theory vastly generalizes the well-known fact that any homomorphism beween Frobenius algebras is an isomorphism, but also allows to go beyond classical (iso)morphisms between Frobenius algebras, especially in finite characteristic, as we show by some explicit examples.

The paper is concluded with some open questions and considerations about Topological Quantum Field Theories.
\end{abstract}

\maketitle

\tableofcontents

\section{Introduction}

There exist many interesting connections between Hopf algebras and Frobenius algebras, or more generally between Hopf structures and Frobenius structures. The most well-known result in this spirit is probably the so-called Larson-Sweedler theorem \cite{LarSwe}, which tells that any finite dimensional Hopf algebra has a non-zero integral and therefore is Frobenius. This result has been refined by Pareigis \cite{Par:HopfFrob}, who showed that a bialgebra is finite dimensional and Hopf if and only if it is Frobenius and the Frobenius structure is in a suitable way compatible with the bialgebra structure (roughly meaning that the Frobenius structure arises from an integral). More generally, Kreimer and Takeuchi showed in \cite{KreTak} that Hopf-Galois objects, or more generally finite Hopf-Galois extensions with free invariants, are Frobenius.
To name just one of the many other links between Hopf and Frobenius conditions, let us mention that it was recently observed by Saracco that a bialgebra $B$ is Hopf if and only if the free functor $-\ot B:{_B\Mod}\to {_B\Mod_B^B}$ is Frobenius (i.e. has an isomorphic left and right adjoint), see \cite{Sar1} and \cite{Sar2}. 

The aim of this paper is to show another relationship between Frobenius algebras and Hopf structures, which is of a very different nature. Our investigations start from the well-known observation that any morphism between two Frobenius algebras (preserving both the algebra and coalgebra structures) is necessarily an isomorphism. This turns the category of Frobenius algebras into a groupoid. And as is also well-known, groupoids give rise to weak (multiplier) Hopf algebras via linearisation. This is the link between "Frobenius" and "Hopf" that we aim to explore, although in a more general setting.

To this end, we will make use of Sweedler's theory of measuring coalgebras \cite{Sweedler}. This theory allows to treat homomorphisms between algebras $A$ and $B$ as well as derivations and various other types of generalized morphisms at the same time. More precisely, one considers a coalgebra $P$ together with a linear map $\psi:P\ot A\to B$ such that the associated map $A\to \Hom(P,B)$ is an algebra morphism with respect to the convolution algebra structure on the codomain.
Moreover, Sweedler showed the existence of a {\em universal} measuring coalgebra between any pair of algebras. This theory has been generalized in several ways, for example to the setting of closed monoidal categories in \cite{Chris} and to more general types of algebraic structures, termed $\Omega$-algebras in \cite{AGV}. 

It was already implicitly observed in Sweedler that the existence of universal measuring coalgebras allows the enrichement of the category of algebras over the category of coalgebras. In the terminology of \cite{BCV} this means that the category of algebras can be turned into a ($k$-linear) semi-Hopf category. Following the work of \cite{AGV}, the same holds for any category of $\Omega$-algebras. A natural question arises whether it is also possible to make this category into a Hopf-category, that is, to show the existence of an antipode for this semi-Hopf category. One possibility is to extend the considered semi-Hopf category by a free or cofree construction. The fact that this is indeed possible has been shown in \cite{GV}. However, more interestingly, one could look for conditions such that the semi-Hopf category arising from the universal measurings {\em is} Hopf. A first result of this flavour has been discussed in \cite{Chris}, where it was shown that the universal measuring coalgebra between a cocommutative Hopf algebra and a commutative bialgebra is itself a Hopf algebra in a natural way. Remark however, that this Hopf algebra structure is ``local'' and hence this construction does not give rise to a Hopf category. 

The main aim of this paper is to show (see Theorem~\ref{univact}) that the semi-Hopf category of universal measurings between Frobenius algebras is automatically Hopf, and that the antipode is moreover invertible. By restricting to the grouplike elements in the measuring coalgebras, one recovers from this the above mentioned result, that any morphism between Frobenius algebras is invertible. 

Dually to measuring coalgebras, one can consider comeasuring algebras. However, universal comeasurings only exist under suitable finiteness conditions (see \cite{AGV}). Since Frobenius algebras are necessarily finite dimensional, universal comeasurings exist as well between any pair of Frobenius algebras. This, in turn, allows to endow a category of Frobenius algebras with a semi-Hopf {\em op}category structure, which again turns out to be Hopf (see Theorem~\ref{univcoact}). As one can expect, both theories are dual to each other, in the sense that the measuring coalgebras arise as the Sweeder dual of comeasuring algebras. Furthermore, also the notion of Frobenius algebra itself has a nice self-duality. This can be combined with the measurings in several ways, which allows to recover the antipode of the measuring Hopf category, as we show in Theorem~\ref{factor}.

To illustrate our theory, we explicitly compute the universal measuring coalgebra for some small examples of Frobenius algebras arising from group algebras and matrix algebras. These examples show some particular behaviour, that motivates us to formulate some open questions for future investigations. In particular, we discuss the remarkable behaviour of universal measuring coalgebras, that existence of an (iso)morphism after base extension is already reflected by the existence of non-zero measurings before base extension (see Proposition~ Proposition~\ref{pr:descent}).
We finish the paper with some consideration about Topological Quantum Field Theory that were in fact the initial motivation for our investigations.

\subsection*{Notations and conventions}

Throughout this paper we work with algebras over a commutative field $k$, i.e.\ within the symmetric monoidal category of $k$-vector spaces that we denote as $\Vect_k$. However, the results of this paper hold in a general symmetric monoidal base category $\Vv$, under the condition that universal (co)measuring objects exist in $\Vv$ (examples of such categories are, apart from vector spaces over field, representations of a group or graded vector spaces over a commutative group, for more examples we refer for example to \cite{AGV2}).

We denote the identity morphism on an object $A$ in $\Vv$ by $A$, $\Id_A$ or just by $\Id$.

\subsection*{Acknowledgements}
We thank William Hautekiet for assistance to compute the dimensions of universal measuring coalgebras by means of computer algebra. The first author was supported by a Fria fellowship of FNRS FNRS (Fonds National de la Recherche Scientifique), grant number FC41285, while working on this project. The second author would like to thank the FWB (fédération Wallonie-Bruxelles) for support through the ARC project ``From algebra to combinatorics, and back''.

\section{Preliminaries and first results}

\subsection{Morphisms of Frobenius algebras}

We recall some definitions and known properties of Frobenius algebras, as they will appear to be useful in the new results obtained later in this paper. We refer to \cite{CMZ} and \cite{Kock} for more details.

\begin{Definition}
A {\em Frobenius algebra} is 5-tuple $(A,\mu,\eta, \Delta, \nu)$, with linear maps
\begin{eqnarray*}
\mu: A\ot A \to A, \qquad \eta: k\to A,\\
\Delta: A\to A\ot A, \qquad \nu: A\to k,
\end{eqnarray*}
such that:
\begin{itemize}
\item $(A,\mu,\eta)$ is an (associative, unital) algebra,
\item $(A,\Delta, \nu)$ is a (coassociative, counital) coalgebra,
\item the {\em Frobenius conditions} are satisfied, meaning that the following diagram commutes
\end{itemize}
\[
\xymatrix{ & A \ot A \ar[d]^\mu \ar[dl]_-{\Delta\ot A} \ar[dr]^-{A\ot \Delta}  \\
A\ot A \ot A \ar[rd]_{A\ot \mu}& A \ar[d]^\Delta & A\ot A \ot A \ar[dl]^{\mu\ot A}  \\
            & A \ot A. }
\] 
A {\em morphism of Frobenius algebras} is a linear map that is an algebra morphism as well as a coalgebra morphism.
\end{Definition}

With Sweedler type notation $\Delta(a)=a^{(1)}\ot a^{(2)}$, the Frobenius conditions tell us that for all $a, b \in A$ :
$$(ab)^{(1)}\ot (ab)^{(2)}=a^{(1)}\ot a^{(2)}b=ab^{(1)}\ot b^{(2)}.$$
Consequently, the comultiplication is completely determined by its value in the unit, which is given by the so-called {\em Casimir element}
$$\Delta(1)=1^{(1)}\ot 1^{(2)}=:e^1\ot e^2=e \in A\ot A$$
which satisfies - because of the bilinearity of the comultiplication - the following Casimir property 
\begin{equation} e^{1}\ot e^{2}a=\Delta(1a)=\Delta(a)=\Delta(a1)=ae^{1}\ot e^{2}. \eqlabel{casimir} \end{equation}
In calculations one often has to deal with more than one copy of the Casimir element. Hence we will use the notation 
$$e^{1}\ot e^{2}=E^{1}\ot E^{2}=\varepsilon^1\ot \varepsilon^2.$$
By the Casimir property, we then have the following identities
$$E^1e^1\ot e^2\ot E^2=e^1\ot e^2E^1\ot E^2=e^1\ot E^1\ot E^2e^2.$$
The counitality of the comultiplication can be expressed in terms of the Casimir element in the following way
$$\nu(e^1)e^2=1=e^1\nu(e^2).$$

It is well known that the coassociativity of the the comultiplication of a Frobenius algebra can be deduced from its associativity (and vice versa). Moreover a Frobenius algebra $A$ is called {\em symmetric} if $\nu(ab)=\nu(ba)$ for all $a,b\in A$, or equivalently the Casimir element satisfies 
$$e^1\ot e^2=e^2\ot e^1.$$
Classical examples of Frobenius algebras are matrix algebras, finite group algebras and finite dimensional Hopf algebras.

It turns out that the restrictions on Frobenius algebras are so strong, that every morphism of Frobenius algebras is already an isomorphism. We give a detailed proof of this very well-known result, since we will generalize the argument later. 

\begin{proposition} \prlabel{frobiso}
Every morphism of Frobenius algebras is an isomorphism.
\end{proposition}
\begin{proof}
Let $A$ and $B$ be two Frobenius algebras. If $h:A\to B$ is morphism of Frobenius algebras, $h$ also preserves the Casimir elements, since
$$h(e^1)\ot h(e^2)=(h\ot h) \Delta(1_A)=\Delta h(1_A)=\Delta(1_B)=:f^{1}\ot f^{2}.$$
We can define an inverse $g:B\to A$ by
$$g(b)=e^1\nu_B(h(e^2)b).$$
Let us check that $g$ and $h$ are mutual inverses:
\begin{eqnarray*}
(hg)(b)&=&h(e^1)\nu_B(h(e^2)b)=f^{1}\nu_B(f^{2}b)=bf^1\nu_B(f^{2})=b,\\
(gh)(a)&=&e^1\nu_B(h(e^2)h(a))=e^1\nu_B(h(e^2a))=e^1\nu_A(e^2a)=ae^1\nu_A(e^2)=a.
 \end{eqnarray*}
Therefore $h$ is an isomorphism of Frobenius algebras.
\end{proof}

\begin{remark}
As usual, we know that the linear inverse of an invertible morphism is a morphism itself. In case of a morphism of Frobenius algebras, it can be verified explicitly that the proposed inverse as in the above proof is a morphism of Frobenius algebras, even without making explicit use of its invertibility. Let us provide the explicit computations here, as they are the guideline for more general computations further in this paper.

The map $g$ preserves the comultiplication:
\begin{eqnarray*}
g(b^{(1)})\ot g(b^{(2)})&=&g(f^1)\ot g(f^2b)=e^1\nu_B(h(e^2)f^1)\ot E^1\nu_B(h(E^2)f^2b)\\
&=&e^1\nu_B(h(e^2)h(\varepsilon^1))\ot E^1\nu_B(h(E^2)h(\varepsilon^2)b)\\
&=&e^1\nu_B(h(e^2\varepsilon^1))\ot E^1\nu_B(h(E^2\varepsilon^2)b)\\
&=&\varepsilon^1e^1\nu_B(h(e^2))\ot E^1\nu_B(h(E^2\varepsilon^2)b)\\
&=&\varepsilon^1e^1\nu_A(e^2)\ot E^1\nu_B(h(E^2\varepsilon^2)b)\\
&=&\varepsilon^1\ot E^1\nu_B(h(E^2\varepsilon^2)b)\\
&=&\varepsilon^1\ot \varepsilon^2E^1\nu_B(h(E^2)b)\\
&=&\Delta(g(b)).
\end{eqnarray*}

The map $g$ preserves the unit:
\begin{eqnarray*}
g(1_B)=e^1_A \nu_B(h(e^2)1_B)=e^1_A \nu_A(e^2)=1_A.
\end{eqnarray*}

The map $g$ preserves the multiplication:
\begin{eqnarray*}
g(b)g(b')&=&e^1\nu_B(h(e^2)b)E^1\nu_B(h(E^2)b')\\
&=&e^1E^1\nu_B(h(e^2)b)\nu_B(h(E^2)b')\\
&=&E^1\nu_B(h(e^2)b)\nu_B(h(E^2e^1)b')\\
&=&E^1\nu_B(h(e^2)b)\nu_B(h(E^2)h(e^1)b')\\
&=&E^1\nu_B(f^2)b)\nu_B(h(E^2)f^1b')\\
&=&E^1\nu_B(f^2))\nu_B(h(E^2)bf^1b')\\
&=&E^1\nu_B(h(E^2)bf^1\nu_B(f^2)b')\\
&=&E^1\nu_B(h(E^2)bb')\\
&=&g(bb').
\end{eqnarray*}

The map $g$ preserves the counit:
\begin{eqnarray*}
\nu_A(g(b))=\nu_A(e^1)\nu_B(h(e^2)b)=\nu_B(h(\nu_A(e^1)e^2)b)=\nu_B(h(1_A)b)=\nu_B(1_B b)=\nu_B(b).
\end{eqnarray*}
\end{remark}

\subsection{Hopf categories and Hopf opcategories}

Hopf $\Vv$-categories, where $\Vv$ is a symmetric (or braided) monoidal category were introduced in \cite{BCV}. We recall the definitions and basic properties for the reader's convenience.

\begin{Definition}\label{defsHcat}
A {\em semi-Hopf $\Vv$-category} where $(\Vv,\ot, I, \sigma)$ is a braided category is a category $\ul A$ enriched over the monoidal category of coalgebras (or comonoids) in $\Vv$. Explicitely, a semi-Hopf $\Vv$-category $\ul A=(A^0,A_{xy})$ consists of a collection\footnote{This could be a proper class although to avoid set-theoretical issues we will mostly suppose that $A^0$ is a set.} of objects $A^0$ and for all objects $x,y\in A^0$ we have an object $A_{xy} \in \Vv$ together with morphisms in $\Vv$ (for all $x,y,z\in A^0$):
\begin{eqnarray*}
m_{xyz}:  A_{xy}\ot A_{yz}\to A_{xz}, \qquad j_x: I \to A_{xx},\\
\delta_{xy}: A_{xy}\to A_{xy} \ot A_{xy}, \qquad \epsilon_{xy}: A_{xy}\to I
\end{eqnarray*}
satisfying the commutativity of the following diagrams.
\[\xymatrix{ 
A_{xy}\ot A_{yz}\ot A_{zu} \ar[rr]^-{m_{xyz}\ot \Id} \ar[d]_-{\Id\ot m_{yzu}} && A_{xz}\ot A_{zu} \ar[d]^-{m_{xzu}} \\
A_{xy}\ot A_{yu} \ar[rr]^-{m_{xyu}} && A_{xu}
}\quad
\xymatrix{
A_{xy}\ot A_{yy} \ar[dr]^-{m_{xyy}} && A_{xx}\ot A_{xy} \ar[dl]_-{m_{xxy}}\\
A_{xy}\ot I \ar[u]^-{\Id \ot j_y} \ar[r]_\cong & A_{xy} & I\ot A_{xy} \ar[u]_-{j_x\ot \Id} \ar[l]^\cong
}
\]
\[
\xymatrix{
A_{xy} \ar[rr]^-{\delta_{xy}} \ar[d]_-{\delta_{xy}} && A_{xy}\ot A_{xy} \ar[d]^-{\delta_{xy}\ot \Id} \\
A_{xy}\ot A_{xy} \ar[rr]_-{\Id \ot \delta_{xy}} && A_{xy}\ot A_{xy} \ot A_{xy} 
}\qquad
\xymatrix{
A_{xy} \ar[drr]^-{\delta_{xy}} \ar[rr]^\cong \ar[d]_\cong && A_{xy}\ot I \\
I\ot A_{xy} && A_{xy}\ot A_{xy} \ar[u]_-{\Id\ot \epsilon_{xy}} \ar[ll]^-{\epsilon_{xy}\ot \Id}
}
\]
\[
\xymatrix{A_{xy}\ot A_{yz} \ar[rr]^-{\delta_{xy}\ot \delta_{xy}} \ar[dd]_{m_{xyz}} & & A_{xy}\ot A_{xy}\ot A_{yz}\ot A_{yz} \ar[d]^{\Id \ot \sigma \ot \Id} \\
& & A_{xy}\ot A_{yz} \ot A_{xy}\ot A_{yz} \ar[d]^{m_ {xyz} \ot m_{xyz}} \\
A_{xz} \ar[rr]_{\delta_{xz}} & & A_{xz}\ot A_{xz} } 
\]
\[
\xymatrix{I \ar[r]^-{\simeq} \ar[d]_{j_x} & I\ot I \ar[d]^{j_x\ot j_x} &A_{xy}\ot A_{yz} \ar[r]^-{\epsilon_{xy}\ot \epsilon_{yz}} \ar[d]_{m_{xyz}} & I\ot I \ar[d]^{\simeq} &I \ar[d]_{j_x} \ar[r]^\Id & I \ar[d]^\Id \\
A_{xx} \ar[r]_-{\delta_{xx}} & A_{xx}\ot A_{xx} &A_{xz} \ar[r]_{\epsilon_{xz}} & I &A_{xx} \ar[r]_{\epsilon_{xx}} & I } 
\]

A morphism of semi-Hopf $\Vv$-categories $\ul f:\ul A\to \ul B$ consists of a map $f^0:A^0\to B^0$ and for each $x,y\in A^0$ a $\Vv$-morphism $f_{xy}: A_{xy}\to B_{fxfy}$\footnote{In order not to overload notation, we denote $f^0(x)$ simply by $fx$ for any $x\in A^0$.} that respect the structure morphisms in the following sense.
\begin{eqnarray*}
\xymatrix{
A_{xy} \ar[rr]^-{f_{xy}} \ar[d]_-{\delta_{xy}} && B_{fxfy} \ar[d]^-{\delta_{fxfy}} \\
A_{xy}\ot A_{xy} \ar[rr]^-{f_{xy}\ot f_{xy}}  && B_{fxfy}\ot B_{fxfy}} 
&& \xymatrix{A_{xy} \ar[dr]_-{\epsilon_{xy}} \ar[rr]^-{f_{xy}} && B_{fxfy} \ar[dl]^-{\epsilon_{fxfy}} \\ & I } \\
\xymatrix{A_{xy}\ot A_{yz} \ar[rr]^-{f_{xy}\ot f_{yz}} \ar[d]_-{m_{xyz}} && B_{fxfy}\ot B_{fyfz} \ar[d]^-{m_{fxfyfz}} \\
A_{xz} \ar[rr]^-{f_{xz}} && B_{fxfz}
}
&&
\xymatrix{& I \ar[dl]_-{j_x} \ar[dr]^-{j_{fx}} \\
A_{xx} \ar[rr]^-{f_{xx}} && B_{fxfx}}
\end{eqnarray*}
In other words, $\ul f$ is a $\Coalg(\Vv)$-functor where $\Coalg(\Vv)$ denotes the category of coalgebras in $\Vv$.
We denote the category of semi-Hopf $\Vv$-categories and their morphisms by $\Vv$-$\sHopf$.
\end{Definition}

From the definition it might be clear that $\ul A=(\{*\},A_{**})$ is a semi-Hopf category with just one object, if and only if $A_{**}$ is a bialgebra in $\Vv$. Conversely, for any semi-Hopf category, $\ul A$ and any object $x\in A^0$, we have that $A_{xx}$ is a bialgebra in $\Vv$ and $A_{xy}$ is an $(A_{xx},A_{yy})$-bimodule coalgebra. The transition from bialgebras to Hopf algebras leads in the ``many-object case'' to the introduction of a Hopf category as in the next definition.

\begin{Definition}\label{defHcat}
A semi-Hopf $\Vv$-category $\ul H$ is called {\em Hopf $\Vv$-category}, if it is equipped with a family of $\Vv$-morphisms $s_{xy}: H_{xy} \to H_{yx}$, called the {\em antipode}, satisfying: 
\begin{eqnarray*}
m_{xyx}(A_{xy}\ot s_{xy})\delta_{xy}&=&j_x \epsilon_{xy},\\
m_{yxy}(s_{xy}\ot A_{xy})\delta_{xy}&=&j_y \epsilon_{xy}. 
\end{eqnarray*}
The category of Hopf $\Vv$-categories and morphisms of semi-Hopf categories between them is denoted by $\Vv$-$\Hopf$.
\end{Definition}

In what follows, when $\Vv=\Vect_k$, we will simply use the term ``Hopf category'' for a Hopf $\Vect_k$-category. 
The reader should be warned that the name ``Hopf category'' is also used for a different notions, as considered in \cite{CraFre} or \cite{Pfe}.

One can show (see \cite{BCV}) that a morphism of semi-Hopf categories between Hopf categories preserves the antipode, hence that the antipode for a semi-Hopf category is unique whenever it exits. Furthermore, the antipode of a Hopf category $\ul A$ satisfies the following properties:
\begin{eqnarray} \label{anti}
&s_{xy} m_{xyz} = m_{zyx} \sigma (s_{xy}\ot s_{yx}) \qquad&\delta_{yx}s_{xy}=\sigma(s_{xy}\ot s_{xy})\delta_{xy}\\
&s_{xx}j_x=j_x \qquad &\epsilon_{yx}s_{xy}= \epsilon_{xy}.
\end{eqnarray}
In other words, $s$ can be viewed as an identity-on-objects morphism of semi-Hopf categories from $\ul A$ to $\ul A^{op,cop}$, where $\ul A^{cop}$ is the semi-Hopf category with co-opposite comultiplications at $A_{xy}$ for every $x,y\in A^0$ and $\ul A^{op}$ is the semi-Hopf category with Hom-objects $A_{xy}^{op}=A_{yx}$ for all $x,y\in A^0$ and composition $\xymatrix{A_{yx}\ot A_{zy} \ar[r]^{\sigma^{-1}} & A_{zy}\ot A_{yx} \ar[r]^-{m_{zyx}}& A_{zx} }$.

Where (semi-) Hopf categories have a "local" coalgebra structure and a "global" algebra structure, one can also consider dual objects with a "local" algebra structure and "global" coalgebra structure. We term such objects "(semi-)Hopf opcategories", as is also done in \cite{BFVV} and \cite{GV}.

\begin{Definition}\label{def:sHopfopcat} \label{defsHopcat}
A {\em semi-Hopf $\Vv$-opcategory} $\ul A=(A^0,A_{xy})$ consists of a collection of objects $A^0$ and for all objects $x,y\in A^0$ we have an object $A_{xy} \in \Vv$ together with morphisms in $\Vv$ (for all $x,y,z\in A^0$):
\begin{eqnarray*}
\mu_{xy}: A_{xy}\ot A_{xy}  \to A_{xy}, \qquad \eta_{xy}: I \to A_{xy},\\
d_{xyz}: A_{xz}\to A_{xy} \ot A_{yz}, \qquad e_{x}: A_{xx}\to I
\end{eqnarray*}
satisfying the commutativity of the following diagrams. 
\[\xymatrix{ 
A_{xu} \ar[rr]^-{d_{xyu}} \ar[d]_-{d_{xzu}} && A_{xz}\ot A_{zu} \ar[d]^-{d_{xyz}\ot \Id} \\
A_{xy}\ot A_{yu} \ar[rr]^-{\Id \ot d_{yzu}} && A_{xy} \ot A_{yz}\ot A_{zu}
}\quad
\xymatrix{
A_{xy}\ot A_{yy} \ar[d]_-{\Id \ot e_y} && A_{xx}\ot A_{xy} \ar[d]^-{e_x\ot \Id} \\
A_{xy}\ot I  \ar[r]_\cong & A_{xy} \ar[ul]_-{d_{xyy}} \ar[ur]^-{d_{xxy}} & I\ot A_{xy}  \ar[l]^\cong
}
\]
\[
\xymatrix{
A_{xy}\ot A_{xy} \ot A_{xy}  \ar[rr]^-{\Id \ot \mu_{xy}} \ar[d]_-{\mu_{xy}\ot \Id } && A_{xy}\ot A_{xy} \ar[d]^-{\mu_{xy}} \\
A_{xy}\ot A_{xy} \ar[rr]_-{\mu_{xy}} && A_{xy}
}\qquad
\xymatrix{
A_{xy}\ot A_{xy} \ar[drr]^-{\mu_{xy}}  && A_{xy}\ot I \ar[ll]_-{\Id \ot \eta_{xy}} \\
I\ot A_{xy} \ar[u]^-{\eta_{xy}\ot \Id}&& A_{xy} \ar[u]_-{\cong} \ar[ll]^-{\cong}
}
\]
\[
\xymatrix{A_{xz}\ot A_{xz} \ar[rr]^-{d_{xyz}\ot d_{xyz}} \ar[dd]_{\mu_{xz}} & & A_{xy}\ot A_{yz}\ot A_{xy}\ot A_{yz} \ar[d]^{\Id \ot \sigma \ot \Id} \\
& & A_{xy}\ot A_{xy} \ot A_{yz}\ot A_{yz} \ar[d]^{\mu_ {xy} \ot \mu_{yz}} \\
A_{xz} \ar[rr]_{d_{xyz}} & & A_{xy}\ot A_{yz} } 
\]
\[
\xymatrix{I \ar[r]^-{\simeq} \ar[d]_{\eta_{xz}} & I\ot I \ar[d]^{\eta_{xy}\ot \eta_{yz}} &A_{xx}\ot A_{xx} \ar[r]^-{e_x \ot e_x} \ar[d]_{\mu_{xx}} & I\ot I \ar[d]^{\simeq} &I \ar[d]_{\eta_{xx}} \ar[r]^\Id & I \ar[d]^\Id\\
A_{xz} \ar[r]_-{\delta_{xyz}} & A_{xy}\ot A_{yz} &A_{xx} \ar[r]_{e_x} & I &A_{xx} \ar[r]_{e_{x}} & I  }
\]
A morphism of semi-Hopf $\Vv$-opcategories $\ul f:\ul A\to \ul B$ consists of a map $f^0:A^0\to B^0$ and for each $x,y\in A^0$ a $\Vv$-morphism $f_{xy}: A_{xy}\to B_{fxfy}$ that respect the structure morphisms in the following sense.
\begin{eqnarray*}
\xymatrix{
A_{xz} \ar[rr]^-{f_{xz}} \ar[d]_-{d_{xyz}} && B_{fxfz} \ar[d]^-{d_{fxfz}} \\
A_{xy}\ot A_{yz} \ar[rr]^-{f_{xy}\ot f_{yz}}  && B_{fxfy}\ot B_{fyfz}} 
&& \xymatrix{A_{xx} \ar[dr]_-{e_{x}} \ar[rr]^-{f_{xx}} && B_{fxfx} \ar[dl]^-{e_{fx}} \\ & I } \\
\xymatrix{A_{xy}\ot A_{xy} \ar[rr]^-{f_{xy}\ot f_{xy}} \ar[d]_-{\mu_{xy}} && B_{fxfy}\ot B_{fxfy} \ar[d]^-{\mu_{fxfy}} \\
A_{xy} \ar[rr]^-{f_{xy}} && B_{fxfy}
}
&&
\xymatrix{& I \ar[dl]_-{\eta_{xy}} \ar[dr]^-{\eta_{fxfy}} \\
A_{xy} \ar[rr]^-{f_{xy}} && B_{fxfy}}
\end{eqnarray*}
We denote the category of semi-Hopf $\Vv$-opcategories and their morphisms by $\Vv$-$\opsHopf$.
\end{Definition}

Again, we have that $\ul A=(\{*\},A_{**})$ is a semi-Hopf opcategory with just one object if and only if $A_{**}$ is a bialgebra in $\Vv$. Conversely, for any semi-Hopf opcategory, $\ul A$ and any object $x\in A^0$, we have that $A_{xx}$ is a bialgebra in $\Vv$ and $A_{xy}$ is an $(A_{xx},A_{yy})$-bicomodule algebra. 

The definition of a semi-Hopf opcategory is exactly dual to the definition of semi-Hopf category, in the sense that a semi-Hopf $\Vv$-opcategory is exactly a semi-Hopf $\Vv^{op}$-category. However, it is important to obsever that the notion of a morphism between semi-Hopf $\Vv$-opcategories is {\em different} from a morphism of semi-Hopf $\Vv^{op}$-categories ! Indeed, a morphism of semi-Hopf $\Vv$-opcategories $\ul f:\ul A\to \ul B$ consists of a map $f:A^0\to B^0$ and a collection of $\Vv$-morphisms $f_{xy}:A_{xy}\to B_{fxfy}$ satisfying axioms. On the other hand, a morphism of $\Vv^{op}$-categories $\ul f:A\to B$ consists of a collection of $\Vv$-morphisms $f_{xy}:B_{fxfy}\to A_{xy}$ satisfying axioms. Therefore the categories $\Vv$-$\opsHopf$ and $\Vv^{op}$-$\sHopf$ are truly different, and the theory of semi-Hopf opcategories can not be deduced from the theory of semi-Hopf categories by direct dualization ! The latter category is what has been called "dual (semi-)Hopf categories" in \cite{BCV}. 
 For this reason, we have to treat semi-Hopf opcategories separately in what follows.

\begin{Definition}
A semi-Hopf $\Vv$-opcategory $\ul H$ is called {\em Hopf $\Vv$-opcategory}, if it is equipped with a family of $\Vv$-morphisms $s_{xy}: H_{xy} \to H_{yx}$, called the {\em antipode} satisfying: \begin{eqnarray*}
\mu_{xy}(A_{xy}\ot s_{yx})d_{xyx}&=&\eta_{xy} e_{x},\\
\mu_{yx}(s_{xy}\ot A_{yx})d_{xyx}&=&\eta_{yx} e_{x}. 
\end{eqnarray*}
The category of Hopf $\Vv$-categories and morphisms of semi-Hopf categories between them is denoted by $\Vv$-$\opHopf$.
\end{Definition}

\subsection{Measuring and comeasuring semi-Hopf categories for $\Omega$-algebras}

\subsubsection{Measuring coalgebras}

Measuring coalgebras and universal measuring coalgebras between algebras were introduced by Sweedler in \cite{Sweedler}, but can be considered for general algebraic structures with structure maps of arbitrary arity, called $\Omega$-algebras, see \cite{AGV}.

\begin{definition}
Let $\Omega$ be a set with two maps $s,t:\Omega\to \NN$ called respectively {\em source} and {\em target}. An $\Omega$-algebra is a vector space $A$ endowed with linear maps $\omega_A:A^{\ot s(\omega)}\to A^{\ot t(\omega)}$ for all $\omega\in \Omega$. A {\em morphism} of $\Omega$-algebras $f:A\to B$ is a linear map satisfying $f^{\ot t(\omega)}\circ \omega_A = \omega_B\circ f^{\ot s(\omega)}$ for all $\omega\in \Omega$. We denote the category of $\Omega$-algebras and morphisms of $\Omega$-algebras by $\Omega$-$\Alg$. Furthermore, given a class $\Xx$ of $\Omega$-algebras we define the full sub-category of $\Omega$-algebras in $\Xx$ by $\Omega$-$\Alg(\Xx)$.
\end{definition}

By considering suitable choices of $(\Omega,s,t)$, algebras, coalgebras, Frobenius algebras and many other algebraic structures can be viewed as $\Omega$-algebras. For example, in case of Frobenius algebras, we have $\Omega=\{\mu,\eta,\Delta,\nu\}$.
Remark that in the definition of $\Omega$-algebra, we do not suppose any coherence conditions (such as associativity of coassociativity) for the considered structure maps. Therefore, it is often useful not to consider the category of all $\Omega$-algebras for a given $\Omega$, but rather to consider a suitable (full) subcategory of it, that is, to consider a suitable class $\Xx$ of $\Omega$-algebras, for example associative algebras with $\Omega=\{\mu,\eta\}$.

\begin{definition}\delabel{measuring}
Let $P$ be a coalgebra, $A$ and $B$ two $k$-vector spaces. We call a linear map $\psi: P \ot A \to B$ a {\em measuring} from $A$ to $B$. By hom-tensor relations, a measuring is equivalent to a linear map $\hat\psi:P\to \Hom(A,B)$. Consequently, we will use the following notation for the image under $\psi$:
$$\psi(p\ot a)=p(a).$$
Furthermore, for any $n\in \NN_0$, we then define and denote $\psi_n:P\ot A^{\ot n}\to B^{\ot n}$ by 
$$\psi_n(p\ot a)=p(a):=p^{(1)}(a_1)\ot \cdots \ot p^{(1)}(a_n) \in B^{\ot n},$$
for any $a=a_1\ot\cdots \ot a_n\in A^{\ot n}$. We also define $\psi_0=\epsilon_P:P\cong P\ot A^{\ot 0}\to B^{\ot 0}\cong k$.

Suppose now that $A$ and $B$ are $\Omega$-algebras. Then a measuring $\psi: P \ot A \to B$ is said to be a {\em measuring of $\Omega$-algebras} if and only if for all $p\in P$, $\omega\in \Omega$ and $a\in A^{\ot s(\omega)}$ we have (with notation introduced above):
$$p(\omega_A(a)) = \omega_B(p(a)) \in B^{\ot t(\omega)}.$$
A measuring coalgebra is denoted as a pair $(P,\psi)$.

If $(P,\psi)$ and $(P',\psi')$ are two measurings from $A$ to $B$, then a morphism of measurings is a coalgebra map $f:P\to P'$ such that $\psi= \psi'\circ (f\ot A)$. Measurings from $A$ to $B$ and their morphisms form a category $\Meas(A,B)$.
\end{definition}

\begin{lemma}\label{le:meascomposition}
\begin{enumerate}[(1)]
\item A linear map $f:A\to B$ is a morphism of $\Omega$-algebras if and only if the canonically associated map $f:k\ot A\cong A\to B$ makes $k$ into a measuring coalgebra from $A$ to $B$.
\item Let $A$, $B$ and $C$ be $\Omega$-algebras, $(P,\psi)$ be a measuring of $\Omega$-algebras from $A$ to $B$ and $(P',\psi')$ be a measuring of $\Omega$-algebras from $B$ to $C$. Then $(P'\ot P,\psi''=\psi' \circ (P'\ot \psi))$ is a a measuring of $\Omega$-algebras from $A$ to $C$.
\item Let $A$ and $B$ be $\Omega$-algebras, let $(P,\psi)$ be a measuring of $\Omega$-algebras from $A$ to $B$ and $f:P'\to P$ a morphism of coalgebras. Then $(P',\psi'=\psi\circ (f\ot A))$ is again a measuring of $\Omega$-algebras.
\item If $(P,\psi)$ is a measuring of $\Omega$-algebras from $A$ to $B$, then for any grouplike element $g\in P$, the map $\psi(g\ot -):A\to B$ is a morphism of $\Omega$-algebras.
\end{enumerate}
\end{lemma}

\begin{proof}
\ul{(1)}. This follows directly from the definitions, taking into account the unique trivial\footnote{By this we mean the unique $k$-coalgebra structure on $k$ having the identity map as counit.} coalgebra structure on $k$.\\
\ul{(2)}. For any $\omega\in \Omega$, any $p\ot q\in P \ot P'$ and any $a\in A^{\ot s(\omega)}$ we find
$$(p\ot q)(\omega_A(a)) = p(q(\omega_A(a))) = p(\omega_B(q(a))) = \omega_C(p(q(a)))=\omega_C((p\ot q)(a)),$$
where the first and last equality follow from the definition of $\psi''$ as in the statement (and the notation introduced above for the image under the measuring map) and the two intermediate equalities express the measuring property of $(P',\psi)'$ and $(P,\psi)$. \\
\ul{(3)}. This follows again by direct computation. For any $\omega\in \Omega$, any $p'\in P'$ and any $a\in A^{\ot s(\omega)}$ we find
$$p'(\omega_A(a))=f(p')(\omega_A(a))=\omega_B(f(p')(a))=\omega_B(p'(a)),$$
where the first and last equality follow from the definition of $\psi'$ as in the statement and the intermediate equality expresses the measuring property of $(P,\psi)$. \\
\ul{(4)}. This follows by combining (2) with (3), taking into account that a $g\in P$ is a grouplike element if and only if $f:k\to P, f(1_k)=g$ is a morphism of coalgebras.
\end{proof}

\begin{remarks}
Similarly to the fact that then any grouplike element $g$ in a coalgebra $P$ measuring from $A$ to $B$ induces a morphism of $\Omega$-algebras from $A$ to $B$ (see (4) in the above lemma), each $g$-primitive element $x$ of $P$ (i.e.\ $\Delta(x)=x\ot g+g\ot x$) induces what one could call a $g$-derivation from $A$ to $B$. 

By item (2) of the above Lemma, given any class $\Xx$ of $\Omega$-algebras, we can build a new category $\Mm(\Xx)$ out of it, whose objects are the elements of $\Xx$, and where for any pair of objects $A,B\in \Xx$, we define $\Hom_{\Mm(\Xx)}(A,B)$ to be the collection of all (isomorphism classes of) measurings $(P,\psi)$ from $A$ to $B$, where composition is defined as in (2). 
By item (1) of the Lemma, the one-dimensional measurings then correspond exactly to the classical morphisms. However, this might lead to some set-theoretical issues, since it is not clear that the collection of all measuring coalgebras from $A$ to $B$ (even up to isomorphism) would form a set. Therefore, we will follow another approach and use universal measurings to define a coalgebra-enriched category, which basically contains the same information as the category described in this remark.
\end{remarks}

Consider two $\Omega$-algebras $A$ and $B$. Let $C(\Hom(A,B))$ be the cofree coalgebra over the $k$-linear Hom-space $\Hom(A,B)$, and denote the associated (universal) projection by $p:C(\Hom(A,B))\to \Hom(A,B)$. Then for any coalgebra $P$ that is measuring from $A$ to $B$, we can consider the associated map $P\to \Hom(A,B), p\mapsto \psi(p\ot -)$. By the universal property of the cofree coalgebra, we therefore obtain a unique morphism of coalgebras $u:P\to C(\Hom(A,B))$ such that $\psi=p\circ u$. On the other hand, we can consider a measuring map 
$$\psi_{AB};\xymatrix{ C(\Hom(A,B)) \ot A \ar[rr]^-{p\ot A} && \Hom(A,B)\ot A \ar[rr]^-{\sf ev} && B},$$
where $\sf ev$ is the usual evaluation map. In general however, $\psi_{AB}$ is not a measuring of $\Omega$-algebras from $A$ to $B$. Nevertheless, in view of the above, we can consider the sum of all 
all (finite dimensional) sub-coalgebras of $C(\Hom(A,B))$, for which the restriction of $\psi_{AB}$ is a measuring of $\Omega$-algebras. The coalgebra obtained in this way {\em is} measuring from $A$ to $B$ as $\Omega$-algebras, and is moreover a terminal object in the category $\Meas(A,B)$, because the map $u$ above factors through this object. Explicitly, the measuring $C(\Hom(A,B))$ is universal in the following sense.
\begin{definition}
Let $A$ and $B$ be $\Omega$-algebras. The {\em universal measuring coalgebra} from $A$ to $B$, denoted $(\Cc(A,B),\psi_{BA})$ is a measuring coalgebra from $A$ to $B$ such that for any (other) measuring coalgebra $(P,\psi)$ from $A$ to $B$, there exists a unique coalgebra morphism $\theta:P\to \Cc(A,B)$ making the following diagram commutative
$$\xymatrix{ P \ot A \ar[r]^\psi \ar@{.>}[d]_{\exists! \theta \ot A} & B \\
            \Cc(A,B) \ot A \ar[ur] _{\psi_{BA}} & }.$$
\end{definition}

An immediate consequence of this definition is the following useful observation.

\begin{lemma}\label{uniquemorphismuniversal}
Consider two $\Omega$-algebras $A$ and $B$ and let $(\Cc(A,B),\psi_{BA})$ the universal measuring coalgebra from $A$ to $B$. Then for any pair of limear maps $f,g:P\to \Cc(A,B)$, where $f$ is a coalgebra morphism, we have that $f=g$ (hence $g$ is a coalgebra morphism as well) if and only if 
$$\psi_{BA}\circ (f\ot A)=\psi_{BA}\circ (g\ot A)$$
\end{lemma}

\begin{proof}
By Lemma~\ref{le:meascomposition} we know that $\psi_{BA}\circ (f\ot A)$ is a measuring, hence $\psi_{BA}\circ (g\ot A)$ is a measuring as well. The statement now follows directly by the universal property of $\Cc(A,B)$.
\end{proof}

The existence of a universal measuring coalgebra for any given pair of $\Omega$-algebras follows from the reasoning above. This was proven in \cite[Theorem 3.10]{AGV} and already in \cite{Sweedler} in case of usual algebras. 
More general situations about the existence of universal measuring coalgebras in general monoidal categories have been treated in \cite{Chris} and \cite{AGV2}, and allow to transfer also the results of the present paper to such a more general setting. However, for clarity we prefer to restrict to the $k$-linear case here. As observed in \cite{Chris}, and already implicitly present in \cite{Sweedler}, the existence of universal measuring coalgebras between all objects of a given category, allows in fact to enrich this category over the category of coalgebras. Reformulated in terms of semi-Hopf categories, this is summarized in the following result.

\begin{proposition}\label{prop:sHcatOmega}
Let $\Xx$ be any class of $\Omega$-algebras, then we can build a semi-Hopf category $\Cc(\Xx)$ as follows. The objects of $\Cc(\Xx)$ are just the elements of $\Xx$. Furthermore, for any pair of objects ($\Omega$-algebras) $A$, $B$ in $\Xx$, we put $\Cc_{B,A}:=\Cc(A,B)$ the universal measuring coalgebra from $A$ to $B$.  

More explicitly, given any triple of $\Omega$-algebras $A,B,C \in \Cc$, there are coalgebra morphisms
$$m_{A,B,C}:\Cc_{A,B} \ot \Cc_{B,C}\to \Cc_{A,C}$$
and 
$$j_A:k\to \Cc_{A,A}$$
satisfying the associativity and unitality conditions as stated in the first two diagrams of Definition~\ref{defsHcat}. In particular, $\Cc_{A,A}$ is a bialgebra for any $A\in \Xx$.

The category $\Omega$-$\Alg(\Xx)$, whose morphisms are $\Omega$-algebra morphisms between objects in $\Xx$, can be recovered from $\Cc(\Xx)$ by restricting the Hom-objects to the grouplike elements in the universal measuring coalgebras.
\end{proposition}

\begin{proof}
The existence of the coalgebra morphisms $m_{A,B,C}$ and $j_A$, as well as the associativity and unitality conditions for them, follow from the universal property of the universal measuring coalgebras. Indeed, by Lemma~\ref{le:meascomposition}(2), it follows that
\[
\xymatrix{
\psi_{ABC}:
\Cc(B,A) \ot \Cc(C,B) \ot C \ar[rr]^-{\Id \ot \psi_{BC}} && \Cc(B,A) \ot B \ar[rr]^-{\psi_{AB}} && A 
}
\]
is a measuring of $\Omega$-algebras from $C$ to $A$. 
By the universal property of $(\Cc(C,A),\psi_{AC})$ there exists a coalgebra map $m_{A,B,C}$ as in the statement of the proposition. Similarly, the isomorphism $k\ot A\cong A$, which is a measuring of $\Omega$-algebras in a trivial way, implies the existence of the coalgebra map $k\to \Cc_{A,A}$ by the universal property of $(\Cc(A,A),\psi_{AA})$. The associativity and counitality axioms then following again by the unicity in the universal properties of the universal measuring coalgebras.

The last statement is a consequence of Lemma~\ref{le:meascomposition}(1). Indeed, by the universal property of $\Cc(A,B)$, there is a bijection between measurings $k\ot B\to A$ and coalgebra morphisms $k\to \Cc(A,B)$. By Lemma~\ref{le:meascomposition}(1), the former is exactly a morphism of $\Omega$ algebras $B\to A$, where as the latter corresponds to a grouplike element in $\Cc(A,B)$.
\end{proof}

 \begin{remark}
As a consequence of the above proposition we deduce that $\Cc(A,A)$ is a bialgebra, given it is an endo-hom object in a semi-Hopf category. Furthermore, it acts on $A$ and this action is universal in the sense that for any other bialgebra $B$ acting on $A$ via some $\psi: B \ot A\to A$, there exists a bialgebra morphism $\theta: B\to \Cc_{AA}$ such that  
$$\xymatrix{ B \ot A \ar[r]^\psi \ar@{.>}[d]_{\exists! \theta \ot A} & A \\
            \Cc(A,A) \ot A \ar[ur] _{\psi_{AA}} & }.$$
This was shown in \cite{AGV}, by proving that the coalgebra morphism $\theta$ given by the universal measuring condition already preserves the algebra structure. We therefore call $\Cc_{AA}$ the {\em universal acting bialgebra} of $A$.
\end{remark}
\subsubsection{Comeasuring algebras}

The dual notion of a measuring coalgebra is a comeasuring algebra. We recall the notion here, however its existence requires some finiteness conditions.

\begin{definition}
Let $Q$ be an algebra, $A$ and $B$ two vector spaces. We call a linear map $\rho: A \to B \ot Q$, denoted by $\rho(a)=a^{[0]}\ot a^{[1]}\in B\ot Q$, a {\em comeasuring} from $A$ to $B$. For any $n\in \NN_0$, we then define and denote $\rho_n:A^{\ot n}\to B^{\ot n} \ot Q$ by 
$$\rho_n(a)=a^{[0]}\ot a^{[1]} := a_1^{[0]}\ot \cdots \ot a_n^{[0]}\ot (a_1^{[1]}\cdots a_n^{[1]}) \in B^{\ot n}\ot Q.$$
for any $a=a_1\ot\cdots \ot a_n\in A^{\ot n}$. Moreover, we define $\rho_0:A^{\ot 0}\cong k \to B^{\ot 0}\ot Q\cong Q$ to be just the unit map of $Q$. 

Suppose now that $A$ and $B$ are $\Omega$-algebras. Then a comeasuring $\rho: A \to B \ot Q$ is said to be a {\em comeasuring of $\Omega$-algebras} if and only if for all $\omega\in \Omega$ and $a\in A^{\ot s(\omega)}$ we have (with notation introduced above):
$$\rho^{t(\omega)}(\omega_A(a)) = (\omega_B\ot Q)(\rho^{s(\omega)}(a)) \in B^{\ot t(\omega)}\ot Q,$$
A comeasuring algebra is denoted as a pair $(Q,\rho)$. Comeasurings from $A$ to $B$ form in a natural way a category denoted by $\Comeas(A,B)$.

A comeasuring algebra of $\Omega$-algebras $(\Aa(A,B),\rho^{BA})$ from $A$ to $B$ is said to be {\em universal} if for any other comeasuring algebra $(Q,\rho)$, there exists a unique algebra morphism $\phi:\Aa(A,B)\to Q$ such that 
\[\xymatrix{ A \ar[r]^\rho \ar[dr] _{\rho^{BA}} & B\ot Q \ar@{.>}[d]^{\exists! B \ot \phi} \\
           &  B \ot \Aa(A,B)}.
           \] 
A univeral comeasuring algebra is an initial object in $\Comeas(A,B)$.
\end{definition}

In general, a universal comeasuring algebra does not always exist, see for example  \cite[Example 4.20]{AGV}. However, in the same paper, sufficient conditions for the existence are given, under a suitable restriction on the "support" of the comeasuring. In particular, if $B$ is finite dimensional, then a universal comeasuring algebra of the form $\Aa(A,B)$ always exists, and can be constructed by quotienting the free algebra over $B^*\ot A\cong \Hom(B,A)$ by an ideal generated by the comeasuring relations (see \seref{examples}) for more details.

The following result can be proven in a dual way as Lemma~\ref{le:meascomposition}.

\begin{lemma}\lelabel{charactersvsmorphisms}
\begin{enumerate}[(1)]
\item A linear map $f:A\to B$ is a morphism of $\Omega$-algebras if and only if the canonically associated map $f:A\to B\cong B\ot k$ makes $k$ into a comeasuring algebra from $A$ to $B$.
\item Let $A$, $B$ and $C$ be $\Omega$-algebras, $(Q,\rho)$ be a comeasuring of $\Omega$-algebras from $A$ to $B$ and $(Q',\rho')$ be a comeasuring of $\Omega$-algebras from $B$ to $C$. Then $(Q'\ot Q,\rho''=(\rho'\ot Q) \rho )$ is a a comeasuring of $\Omega$-algebras from $A$ to $C$.
\item Let $A$ and $B$ be $\Omega$-algebras, let $(Q,\rho)$ be a comeasuring of $\Omega$-algebras from $A$ to $B$ and $f:Q\to Q'$ a morphism of algebras. Then $(Q',\rho'=(A\ot f)\circ \rho)$ is again a comeasuring of $\Omega$-algebras.
\item If $(Q,\rho)$ is a comeasuring of $\Omega$-algebras from $A$ to $B$, then for an algebra morphism $f:Q\to k$ (i.e. a {\em character}), the map $\hat f: A\to B, \hat f(a)=a^{[0]}f(a^{[1]})$ is a morphism of $\Omega$-algebras. 
\end{enumerate}
\end{lemma}

\begin{proposition}\label{prop:sHopcatOmega}
Let $\Xx$ be any class of $\Omega$-algebras, such that for any pair of $\Omega$-algebras $A, B$ in $\Xx$, the universal comeasuring algebra $(\Aa(A,B),\rho^{BA})$ exists. 

Then we can build a semi-Hopf opcategory $\Aa(\Xx)$ as follows. The objects  of $\Aa(\Xx)$ are the elements of $\Xx$. Furthermore, for any pair of objects ($\Omega$-algebras) $A$, $B$ in $\Xx$, we put $\Aa_{A,B}=\Aa(B,A)$, the universal comeasuring algebra from $B$ to $A$. 

More explicitly, given any triple of $\Omega$-algebras $A,B,C \in \Cc$, there are algebra morphisms
$$d_{A,B,C}:\Aa_{A,C} \to \Aa_{A,B}\ot \Aa_{B,C}$$
and 
$$e_A:\Aa_{A,A}\to k$$
satisfying the coassociativity and counitality conditions as stated in the first two diagrams of Definition~\ref{defsHopcat}. In particular, $\Aa_{A,A}$ is a bialgebra for any $A$ in $\Xx$.

The category $\Omega$-$\Alg(\Xx)$, whose morphisms are $\Omega$-algebra morphisms between objects in $\Xx$, can be recovered from $\Aa(\Xx)$ by replacing each Hom-object $\Aa(B,A)$ by its set of characters (i.e.\ the set of algebra morphisms $\Hom_{\Alg}(\Aa(A,B),k)$).
\end{proposition}

 \begin{remark}
Similarly to the universal acting bialgebra $\Cc(A,A)$ the endo-hom objects of $\Aa(\Xx)$ are bialgebras $\Aa(A,A)$ with a universal coaction. For any other bialgebra $B$ coacting on $A$ via some $\rho: A\to A\ot B$, there exists a bialgebra morphism $\phi: B\to \Aa_{AA}$ such that  
\[\xymatrix{ A \ar[r]^\rho \ar[dr] _{\rho^{AA}} & A\ot B \ar@{.>}[d]^{\exists! A \ot \phi} \\
           &  A \ot \Aa(A,A)},
           \] 
as shown in \cite{AGV}. We therefore call $\Cc_{AA}$ the {\em universal coacting bialgebra} of $A$. For 
\end{remark}

\begin{remark}
Recall that when the considered $\Omega$-algebras are finite dimensional, then $\Aa(A,B)$ is constructed as a quotient of the free algebra over $\Hom(B,A)$. Consequently, there is a canonical map $\Hom(B,A)\to \Aa(A,B)$. Dualizing this map, we obtain a map $\Aa(A,B)^*\to \Hom(B,A)^*\cong \Hom(A,B)$ (recall that $A$ and $B$ are finite dimesnional). If we restrict then to those elements in $\Aa(A,B)^*$ that are algebra morphisms (i.e.\ characters), then the corresponding images are exactly the $\Omega$-algebra maps from $A$ to $B$, and this is exactly the correspondence described in the last statement of the above proposition.
\end{remark}

\section{The Hopf category of Frobenius algebras}\label{HopfFrob}

In this section, we prove the main result of this paper.
In general, the semi-Hopf category of universal measurings from Proposition~\ref{prop:sHcatOmega} is not Hopf. In this case, one can consider the Hopf envelope as constructed in \cite{GV}. However, we will now show that in the case of Frobenius algebras, there exists already an antipode on the category semi-Hopf category of universal measurings of Frobenius algebras, turning it into a Hopf category.
Similarly, the semi-Hopf opcategory of universal comeasurings from Proposition~\ref{prop:sHopcatOmega} will be shown to be Hopf.

Let us first make the definition of measurings between Frobenius algebras explicit.

\begin{definition}
Let $A$ and $B$ be Frobenius algebras and denote their Casimir elements respectively by $e=e^1\ot e^2$ and $f=f^1\ot f^2$. Then a measuring from $A$ to $B$ is a coalgebra $(P,\delta, \epsilon)$ endowed with a linear map $P\ot A\to B$, $p\ot a\mapsto p(a)$, satisfying the following conditions for all $a,a'\in A$ and $p\in P$:
\begin{eqnarray}
p(aa')=p^{(1)}(a)p^{(2)}(a'),& \qquad& p(1_A)=\epsilon(p)1_B\eqlabel{measalg}\\
p(a)^{(1)}\ot p(a)^{(2)}=p^{(1)}(a^{(1)})\ot p^{(2)}(a^{(2)}), &\qquad& \nu_B(p(a))=\epsilon(p)\nu_A(a).\eqlabel{meascoalg}
\end{eqnarray}
Remark that applying the first equality of \equref{meascoalg} to $1_A$, and combining it with the second equality of \equref{measalg} implies
\begin{equation}\eqlabel{measfrob}
\epsilon(p)f^1\ot f^2=p^{(1)}(e^{1})\ot p^{(2)}(e^{2}),
\end{equation}
\end{definition}

As usual, a grouplike element $g\in P$ acts as an automorphism of Frobenius algebras:
\begin{eqnarray}
g(aa')=g(a)g(a'),& \qquad& g(1_A)=1_B\\
g(a)^{(1)}\ot g(a)^{(2)}=g(a^{(1)})\ot g(a^{(2)}), &\qquad& \nu_B(g(a))=\nu_A(a).
\end{eqnarray}
On the other hand, fixing a grouplike $g$ in $P$, a $g$-primitive element $x\in P$ acts as a {\em $g$-bi-derivation}:
\begin{eqnarray}
x(aa')=x(a)g(a') + g(a)x(a'),& \qquad& x(1_A)=0_B\\
x(a)^{(1)}\ot x(a)^{(2)}=x(a^{(1)})\ot g(a^{(2)})+g(a^{(1)})\ot x(a^{(2)}), &\qquad& \nu_B(x(a))=0.
\end{eqnarray}

\begin{theorem} \label{univact}
Let $\Xx$ be any class of Frobenius algebras, and consider the associated semi-Hopf category $\Cc(\Xx)$ as in Proposition~\ref{prop:sHcatOmega}. Then $\Cc(\Xx)$ is Hopf and its antipode is bijective. 

In particular, the universal acting bialgebra $\Cc_{AA}$ on a Frobenius algebra $A$ is a Hopf algebra with bijective antipode.
\end{theorem}

\begin{proof}
Let $A$ and $B$ be two Frobenius algebras in $\Xx$. We want to construct the components $S_{BA}:\Cc_{BA} \to \Cc_{AB}$ of the antipode. Denote by $P=\Cc_{BA}^{cop}$ the co-opposite coalgebra of $\Cc_{BA}$ and the universal measuring $\psi_{BA}: \Cc_{BA}\ot A\to B$. We denote the Casimir element of $A$ by $e^1\ot e^2= E^1 \ot E^2= \varepsilon^1\ot \varepsilon^2$ and the Casimir element of $B$ by $f^1\ot f^2$. We show that $P$ is equipped with a measuring from $B$ to $A$ by means of the map $\bar\psi_{AB}:P\ot B\to A$ given for all $p \in P$ and $b\in B$ by:
$$\bar\psi_{AB}(p\ot b)=\bar p(b):=e^1 \nu_B( \psi_{BA}(p \ot e^2)b)=e^1 \nu_B( p(e^2)b),$$
 where we make use of the notation introduced in \deref{measuring} for measurings.
Let us check that this is indeed a measuring of Frobenius algebras, i.e. that conditions \equref{measalg}-\equref{meascoalg} are satisfied. 
\begin{eqnarray*}
\bar p^{(2)}(b^{(1)})\ot \bar p^{(1)}(b^{(2)})&=&\bar p^{(2)}(f^1)\ot \bar p^{(1)}(f^2b)
=e^1\nu_B(p^{(2)}(e^2)f^1)\ot E^1\nu_B(p^{(1)}(E^2)f^2b)\\
&=&e^1\nu_B(p^{(2)}(e^{2})\epsilon(p^{(3)})f^1)\ot E^1\nu_B(p^{(1)}(E^2)f^2b)\\
&=&e^1\nu_B(p^{(2)}(e^2)p^{(3)}(\varepsilon^1))\ot E^1\nu_B(p^{(1)}(E^2)p^{(4)}(\varepsilon^2)b)\\
&=&e^1\nu_B(p^{(2)}(e^2\varepsilon^1))\ot E^1\nu_B(p^{(1)}(E^2)p^{(3)}(\varepsilon^2)b)\\
&=&\varepsilon^1e^1\nu_B(p^{(2)}(e^2))\ot E^1\nu_B(p^{(1)}(E^2)p^{(3)}(\varepsilon^2)b)\\
&=&\varepsilon^1e^1\nu_A(e^2)\ot E^1\nu_B(p^{(1)}(E^2)p^{(2)}(\varepsilon^2)b)\\
&=&\varepsilon^1\ot E^1\nu_B(p(E^2\varepsilon^2)b)\\
&=&\varepsilon^1\ot \varepsilon^2E^1\nu_B(p(E^2)b)\\
&=&\varepsilon^1\ot \varepsilon^2\bar p(b)\\
&=&(\bar p(b))^{(1)}\ot (\bar p(b))^{(2)},
\end{eqnarray*}
where we used counitality of $P$ in line 2, \equref{measfrob} in line 3, the first equation of \equref{measalg} in line 4, the Casimir property for $e^1\ot e^2$ in line 5, the second equation of \equref{meascoalg}and the conunitality of $P$ in line 6, again \equref{measalg} and $e^1\nu(e^2)=1_A$ in line 7, the Casimir property for $E^1\ot E^2$ in line 8 and finally the definition of $\bar p$ in line 9. Further for the counits we have
\begin{eqnarray*}
\nu_A(\bar p (b))=\nu_A(e^1)\nu_B(p(e^2)b)=\nu_B(p(\nu_A(e^1)e^2)b)=\nu_B(p(1_A)b)=\nu_B(\epsilon(p)1_Bb)=\epsilon(p)\nu_B(b),
\end{eqnarray*}
hence $\bar \psi_{AB}$ satisfies \equref{meascoalg}. We show that it satisfies \equref{measalg}.
\begin{eqnarray*}
\bar p^{(2)}(b) \bar p^{(1)}(b')&=&e^1\nu_B(p^{(2)}(e^2)b)E^1\nu_B(p^{(1})(E^2)b')\\
&=&e^1E^1\nu_B(p^{(2)}(e^2)b)\nu_B(p^{(1)}(E^2)b')\\
&=&E^1\nu_B(p^{(2)}(e^2)b)\nu_B(p^{(1)}(E^2e^1)b')\\
&=&E^1\nu_B(p^{(3)}(e^2)b)\nu_B(p^{(1)}(E^2)p^{(2)}(e^1)b')\\
&=&E^1\nu_P(p^{(2)}) \nu_B(f^2b)\nu_B(p^{(1)}(E^2)f^1b')\\
&=&E^1\nu_B(f^2)\nu_B(\epsilon(p^{(2)})p^{(1)}(E^2)bf^1b')\\
&=&E^1\nu_B(f^2)\nu_B(p(E^2)bf^1b')\\
&=&E^1\nu_B(p(E^2)b\nu_B(f^2)f^1b')\\
&=&E^1\nu_B(p(E^2)bb')\\
&=&\bar p(bb'),
\end{eqnarray*}
where we used the Casimir property for $E^1\ot E^2$ in line 3, the first equation of \equref{measalg} in line 4, \equref{measfrob} in line 5, the Casimir property for $f^1\ot f^2$ in line 6, counitality of $P$ in line 7 and $\nu_B(f^2)f^1=1_B$ in line 9. Finally for the units we have
\begin{eqnarray*}
\bar p (1_B)=e^1\nu_B(p(e^2)1_B)=\epsilon(p) e^1\nu_A(e^2)=\epsilon(p) 1_A.
\end{eqnarray*}
Hence, $\bar \psi_{AB}$ is a measuring of Frobenius algebras and by universal property of $\Cc_{AB}$, there exists a unique map $S_{BA}:P=\Cc_{BA}^{cop}\to \Cc_{AB}$, such that $$\xymatrix{ P \ot B \ar[r]^{\bar\psi_{AB}} \ar@{.>}[d]_{\exists! S_{BA} \ot B} & A \\
            \Cc_{AB} \ot B \ar[ur] _{\psi_{AB}} & }.$$ \\
We show that the morphisms $S_{AB}$ satisfy the antipode conditions for Hopf categories as stated in Definition~\ref{defHcat}. Let $p\in \Cc_{BA}, a\in A, b\in B$ and denote the coalgebra structure on $\Cc_{BA}$ by $\delta_{BA}$ and $\epsilon_{BA}$. Then 
\begin{eqnarray*} 
\bar \psi_{AB}(p^{(1)}\ot \psi_{BA}( p^{(2)}\ot a))&=&e^1\nu_B(p^{(1)}(e^2)p^{(2)}(a))=e^1\nu_B(p(e^2 a))\\
&=&ae^1\nu_B(p(e^2))=a e^1\epsilon_{BA}(p)\nu_A(e^2)=\epsilon_{BA}(p) a \nu_A(e^2) e^1 \\
&=&\epsilon_{BA}(p) \Id_A(a)
\end{eqnarray*}
and conversely
\begin{eqnarray*}
\psi_{BA}(p^{(1)}\ot \bar \psi_{AB}( p^{(2)}\ot b))&=&p^{(1)}(e^1) \nu_B(p^{(2)}(e^2)b)=\epsilon_{BA}(p)f^1\nu_B(f^2b)=\epsilon_{BA}(p) \Id_B(b).
\end{eqnarray*}
This means 
\begin{eqnarray*}
\psi_{AA}(m_{ABA}\ot A)(S_{BA} \ot \Cc_{BA} \ot A)(\delta_{BA}\ot A)&=&\psi_{AB}(\Cc_{AB} \ot \psi_{BA})(S_{BA}\ot \Cc_{BA} \ot A)(\delta_{BA}\ot A)\\
&=&\psi_{AB}(S_{BA}\ot B)(\Cc_{BA} \ot \psi_{BA}) (\delta_{BA}\ot A)\\
&=&\bar \psi_{AB} (\Cc_{BA} \ot \psi_{BA})(\delta_{BA} \ot A)\\
&=&\psi_{AA} (j_{A}\epsilon_{BA} \ot A),
\end{eqnarray*} 
and similarly $\psi_{BB}(m_{BAB}\ot B)(\Cc_{BA} \ot S_{BA} \ot B)(\delta_{BA}\ot B)=\psi_{BB}(j_{B}\epsilon_{BA} \ot B)$. By Lemma~\ref{uniquemorphismuniversal}, we then obtain the antipode condition.

We now show that $S$ is bijective. Denote $P=\Cc_{AB}^{cop}$ and the Casimir element of $B$ by $f^1\ot f^2 = F^1 \ot F^2$. Define $\psi'_{BA}: P \ot A \to B$ by 
$$\psi_{BA}'(p\ot a):= f^2 \nu_A(a \psi_{AB}(p\ot f^1))=f^2 \nu_A(a p(f^1)).$$ We will show that this is a measuring of Frobenius algebras. Firstly, 
\begin{eqnarray*}
 \psi'_{BA}(p^{(2)}\ot a^{(1)})\ot \psi'_{BA}( p^{(1)}\ot a^{(2)})&=&
f^2\nu_A(a^{(1)} p^{(2)}(f^1))\ot F^2\nu_A(a^{(2)}p^{(1)}(F^1))\\
&=&f^2\nu_A(ae^1 p^{(2)}(f^1))\ot F^2\nu_A(e^2p^{(1)}(F^1))\\
&=&f^2\nu_A(ap^{(1)}(F^1)e^1 p^{(2)}(f^1))\ot F^2\nu_A(e^2)\\
&=&f^2\nu_A(ap^{(1)}(F^1)p^{(2)}(f^1))\ot F^2\\
&=&f^2\nu_A(ap(F^1f^1))\ot F^2\\
&=&f^2F^1\nu_A(ap(f^1))\ot F^2\\
&=&\psi'_{BA}(p \ot a)^{(1)} \ot \psi'_{BA}(p \ot a)^{(2)} ,
\end{eqnarray*}
where we used the characterisation of the comultiplication in $A$ via the Casimir element $e^1\ot e^2$ in line 2, the Casimir property of $e^1\ot e^2$ in line 3, $\nu_A(e^2)e^1=1_A$ in line 4, the first equation of \equref{measalg} in line 5, the Casimir property of $f^1\ot f^2$ in line 6 and the characterisation of the comultiplication in $B$ via the Casimir element $F^1\ot F^2$ in line 7. Further we show
\begin{eqnarray*}
\nu_B(\psi'_{BA}(p  \ot a))=\nu_B(f^2) \nu_A(ap(f^1)) =\nu_A(ap(1_B))=
\nu_A(a\epsilon(p)1_A)=\epsilon(p)\nu_A(a),
\end{eqnarray*}
where we used $\nu_B(f^2)f^1=1_B$ and the second equation in \equref{meascoalg}. This shows that $\psi'$ is a measuring of coalgebras. Moreover,
\begin{eqnarray*}
\psi'_{BA}(p^{(2)}\ot a) \psi'_{BA}(p^{(1)}\ot a')&=&f^2\nu_A(ap^{(2)}(f^1))F^2\nu_A(a'p^{(1})(F^1))\\
&=&f^2\nu_A(ap^{(2)}(F^2f^1))\nu_A(a'p^{(1)}(F^1))\\
&=&f^2\nu_A(a(p(f^1))^{(2)})\nu_A(a'(p(f^1))^{(1)})\\
&=&f^2\nu_A(ae^2p(f^1))\nu_A(a'e^1)\\
&=&f^2\nu_A(ae^2a'p(f^1))\nu_A(e^1)\\
&=&f^2\nu_A(aa'p(f^1))\\
&=&\psi'_{BA}( p\ot a'a),
\end{eqnarray*}
where we used the Casimir property of $f^1 \ot f^2$ in line 2, the characterisation of the comultiplication in $B$ via the Casimir element $F^1\ot F^2$ and the first equation of \equref{meascoalg} in line 3,  the characterisation of the comultiplication in $A$ via the Casimir element $e^1\ot e^2$ in line 4, the Casimir property of $e^1\ot e^2$ in line 5 and $\nu_A(e^1)e^2=1_A$ in line 6. Finally the second equation in \equref{meascoalg} and $\nu_B(f^1)f^2=1_B$ yield
\begin{eqnarray*}
\psi'_{BA}(p \ot 1_A)=f^2\nu_B(1_Ap(f^1))=\epsilon(p) f^2\nu_B(f^1)=\epsilon(p) 1_B,
\end{eqnarray*}
therefore $\psi'_{BA}$ is a measuring of Frobenius algebras. 
 Hence there exists a map $S'_{AB}:\Cc_{AB}^{cop} \to \Cc_{BA}$ with 
$$\xymatrix{ P \ot A \ar[r]^{\psi'_{BA}} \ar@{.>}[d]_{\exists! S'_{AB} \ot A} & B \\
            \Cc_{BA} \ot A \ar[ur] _{\psi_{BA}} & }.$$ \\ 
Let $p \in \Cc_{BA}, a\in A$, then:
\begin{eqnarray*}
\psi_{BA} (S'_{AB}S_{BA} \ot A)(p\ot a)&=&\psi'_{BA}(S_{BA}(p)\ot a)\\
&=&f^2 \nu_A(a \psi_{AB}(S_{BA}\ot B)(p\ot f^1))=f^2 \nu_A (a \bar \psi_{AB}(p\ot f^1))\\
&=&f^2 \nu_A (a e^1  \nu_B(p(e^2)f^1))=f^2 \nu_B(p(e^2a)f^1) \nu_A(e^1)\\
&=&f^2 \nu_B(p(\nu_A (e^1)e^2a)f^2) =f^2 \nu_B(p(a)f^1)\\
&=&f^2 p(a) \nu_B(f^1)\\
&=&p(a)=\psi_{BA}(p\ot a)
\end{eqnarray*}
and by universal property of $\Cc_{BA}$ we find $S'_{AB}S_{BA}=\Cc_{BA}$. Analogously for $p \in \Cc_{AB}, b\in B$,
\begin{eqnarray*}
\psi_{AB} (S_{BA}S'_{AB} \ot B)(p\ot b)&=&\bar \psi_{AB}(S'_{AB}(p)\ot b)\\
&=&e^1 \nu_B(\psi_{BA}(S'_{AB}\ot A)(p\ot e^2)b)=e^1 \nu_B(\psi'_{BA}(p\ot e^2)b)\\
&=&e^1 \nu_B(f^2  \nu_A(e^2p(f^1))b)=e^1 \nu_B(f^2 b) \nu_A(e^2p(f^1))\\
&=&p(f^1) e^1 \nu_B(f^2b) \nu_A(e^2)=p(f^1) \nu_B(f^2b)\\
&=&p(bf^1) \nu_B(f^2)\\
&=&p(b)=\psi_{AB}(p\ot b).
\end{eqnarray*}
Therefore, $S_{BA}S'_{AB}=\Cc_{AB}$ by universal property of $\Cc_{AB}$.
\end{proof}

\begin{corollary}
For $A,B$ symmetric Frobenius algebras $S_{AB}^{-1}=S_{BA}$. Particularly $\Cc_{AA}$ is a Hopf algebra with involutive antipode. 
\end{corollary}

\begin{proof}
For symmetric Frobenius algebras we have that $\nu(ab)=\nu(ba)$ and dually $e^1\ot e^2=e^2\ot e^1$. Hence the measuring $\psi'_{AB}$ constructed in the proof of Proposition~\ref{univact} is the same as $\bar \psi_{AB}$ hence by universal property $S_{BA}^{-1}=S_{AB}$.
\end{proof}

Let us now state 
the analogous result for comeasurings between Frobenius algebras.

\begin{definition}
Let $A$ and $B$ be Frobenius algebras and denote their Casimir elements respectively by $e=e^1\ot e^2$ and $f=f^1\ot f^2$. Then a comeasuring from $A$ to $B$ is an algebra $Q$ endowed with a linear map $A\to B\ot Q$, denoted $\rho(a)=a^{[0]}\ot a^{[1]}$, satisfying the following conditions for all $a \in A$:
\begin{eqnarray}
(aa')^{[0]}\ot (aa')^{[1]} = a^{[0]}a'^{[0]}\ot a^{[1]}a'^{[1]} &\quad & (1_A)^{[0]}\ot (1_A)^{[1]} = 1_B\ot 1_Q \eqlabel{comeasalg}\\
\eqlabel{comeas1}a^{(1)[0]}\ot a^{(2)[0]}\ot a^{(1)[1]} a^{(2)[1]}= a^{[0](1)}\ot a^{[0](2)}\ot a^{[1]}, &&
\nu_A(a)1_Q= \nu_B (a^{[0]})a^{[1]}.\eqlabel{comeas2}
\end{eqnarray}
Remark that applying the first equality of \equref{comeas1} to $1_A$ and combining it with the second equality of \equref{comeasalg} leads to
\begin{equation}\eqlabel{comeasfrob2}
e^{1[0]} \ot e^{2[0]} \ot e^{1[1]}e^{2[1]} = f^1\ot f^2 \ot 1_Q.
\end{equation}
\end{definition}

\begin{theorem} \label{univcoact}
Let $\Xx$ be any class of Frobenius algebras, and consider the associated semi-Hopf opcategory $\Aa(\Xx)$ as in Proposition~\ref{prop:sHopcatOmega}. Then $\Aa(\Xx)$ is Hopf and its antipode is bijective. 

In particular, the universal coacting bialgebra $\Aa_{AA}$ on a Frobenius algebra $A$ is a Hopf algebra with bijective antipode.

For $A,B$ symmetric Frobenius algebras we have that the antipode of the Hopf category $\Aa(\Xx)$ satisfies $S_{AB}^{-1}=S_{BA}$. Particularly, $\Aa_{AA}$ is a Hopf algebra with involutive antipode if $A$ is a symmetric Frobenius algebra. 
\end{theorem}

\begin{proof}
We want to construct the components $S_{AB}: \Aa_{AB}\to A_{BA}^{op}$ of the antipode for the semi-Hopf opcategory $\Aa(\Xx)$. Denote by $Q=\Aa_{BA}^{op}$ the opposite algebra of $\Aa_{BA}$. We show that this algebra is equipped with a comeasuring $\bar \rho^{AB}: B\to A\ot Q$, which for $b\in B$  defined by 
$$\bar \rho^{AB} (b):= e^1 \nu_B(e^{2[0]} b) \ot e^{2[1]} \in A\ot Q$$
where $\rho^{BA}(e^2)=e^{2[0]} \ot e^{2[1]}\in B\ot Q$.\\
Let us check that this is indeed a comeasuring. We denote the Casimir element of $A$ by $e^1\ot e^2= E^1 \ot E^2=  \varepsilon^1\ot \varepsilon^2$ and the Casimir element of $B$ by $f^1\ot f^2$. First we show that $\bar \rho^{AB} : B \to A \ot Q$ is an algebra morphism. Let $b,b'\in B$, then 
\begin{eqnarray*}
\bar \rho^{AB} (b) \bar \rho^{AB}(b')&=&(\nu_B(E^{1[0]} b) E^2\ot E^{1[1]})(\nu_B(e^{1[0]} b') e^2\ot e^{1[1]})\\
&=&E^2e^2\nu_B(E^{1[0]}b)\nu_B(e^{1[0]}b')\ot e^{1[1]}E^{1[1]}\\
&=&E^2\nu_B(e^{2[0]}b)\nu_B(E^{1[0]} e^{1[0]} b')\ot E^{1[1]} e^{1[1]} e^{2[1]}\\
&=&E^2\nu_B((1_{A)})^{[0](2)}b)\nu_B(E^{1[0]} (1_{A})^{[0](1)}b')\ot E^{1[1]} 1^{[1]}_{A}\\
&=&E^2\nu_B(1_B^{(2)}b)\nu_B(E^{1[0]} 1_B^{(1)}b')\ot E^{1[1]} 1_Q\\
&=&E^2\nu_B(f^2b)\nu_B(E^{1[0]} f^1b')\ot E^{1[1]}\\
&=&E^2\nu_B(f^2)\nu_B(E^{1[0]} b f^1 b')\ot E^{1[1]}\\
&=&E^2\nu_B(E^{1[0]} b \nu_B(f^2)f^1 b')\ot E^{1[1]}\\
&=&E^2\nu_B(E^{1[0]} bb')\ot E^{1[1]}\\
&=&\bar \rho^{AB} (bb')
\end{eqnarray*}
where we build the product in the in $B\ot \Aa_{BA}^{op}$ in line 2, we used the Casimir property for $e^1\ot e^2$ and the first equation in \equref{comeasalg} in line 3, we used \equref{comeasfrob2} in line 4, the Casimir property for $f^1\ot f^2$ in line 5, $\nu_B(f^2)f^1=1_B$ in line 7 and finally the definition of $\bar \rho^{AB}$ in line 8.
The unitality follows from \equref{comeas2} applied on $e^2$ and $\nu_A(e^1) e^2=1_A$:
\begin{eqnarray*}
\bar \rho^{AB}(1_B)=e^2\ot \nu_B(e^{1[0]}) e^{1[1]}= \nu_B(e^1)e^2\ot  1_Q= 1_A\ot 1_Q
\end{eqnarray*}

We show the comeasuring conditions for the coalgebras $A$ and $B$ for $b\in B$. Denote the multiplication in $Q$ by $m_Q:Q\ot Q\to Q$, then:
\begin{eqnarray*}
&&(A\ot A\ot \mu_Q)(A\ot \sigma \ot Q)(\bar\rho^{AB} \ot \bar\rho^{AB})\Delta_B(b)\\
&=&(A\ot A\ot \mu_Q)(A\ot \sigma \ot Q)(\bar\rho^{AB} \ot \bar\rho^{AB})(f^1\ot f^2b)\\
&=&\nu_B(e^{1[0]} f^1) e^2\ot \nu_B(E^{1[0]} f^2 b) E^2\ot E^{1[1]} e^{1[1]}\\
&=&\nu_B(e^{1[0]} \varepsilon^{2[0]}) e^2\ot \nu_B(E^{1[0]} \varepsilon^{2[0]}b) E^2\ot E^{1[1]}e^{1[1]} \varepsilon^{1[1]}\varepsilon^{2[1]}\\
&=&\nu_B((e^1 \varepsilon^1)^{[0]}) e^2\ot \nu_B(E^{1[0]} \varepsilon^{2[0]}b) E^2\ot E^{1[1]}(e^1\varepsilon^1)^{[1]}\varepsilon^{2[1]}\\
&=&e^2\ot \nu_B(E^{1[0]} \varepsilon^{2[0]}b) E^2\ot E_{1[1]} \nu_B((e^1 \varepsilon^1)^{[0]}) (e^1\varepsilon^1)^{[1]}\varepsilon^{2[1]}\\
&=&e^2\ot \nu_B(E^{1[0]} \varepsilon^{2[0]}b) E^2\ot E^{1[1]} \nu_A(e^1 \varepsilon^1) \varepsilon^{2[1]}\\
&=&e^2 \varepsilon^2\ot \nu_B(E^{1[0]}e^{1[0]} b) E^2\ot E^{1[1]} \nu_A(\varepsilon^1) e^{1[1]}\\
&=&e^2 \nu_A(\varepsilon^1) \varepsilon^2\ot \nu_B((E^1 e^1)^{[0]} b) E^2\ot (E^1  e^1)^{[1]}\\
&=&e^2 \ot \nu_B((E^1 e^1)^{[0]} b) E^2\ot (E^1  e^1)^{[1]}\\
&=&e^2E^1 \ot \nu_B(e^{1[0]} b) E^2\ot e^{1[1]}\\
&=&e^2\ot \nu_B(e^{1[0]} b) e^3\ot e^{1[1]}=(\Delta_A\ot Q)\bar \rho^{AB} (b),
\end{eqnarray*}

where we used the characterisation of comultiplication in $B$ via the Casimir element $f^1 \ot f^2$ in line 2, \equref{comeasfrob2} in line 3, the first equation of \equref{comeasalg} in line 4, \equref{comeas2} in line 6, the Casimir property of $e^1\ot e^2$ in line 7, $\nu_A(e^2)e^1=1_A$ in line 8 and the characterisation of comultiplication in $A$ via the Casimir element $\varepsilon^1 \ot \varepsilon^2$ in line 9.
Further, using \equref{comeas2}, \equref{comeasfrob2} and $\nu_B(f^1)f^2=1_B$ we have: 
\begin{eqnarray*}
(\nu_A\ot Q) \bar \rho^{AB} (b)&=&\nu_B(e^{1[0]}b)\nu_A(e^2) 1_A e^{1[1]}=\nu_B(e^{1[0]}b)\nu_B(e^{2[0]})e^{2[1]}e^{1[1]}\\
&=&\nu_B(f^1 b)\nu_B(f^2)1_Q=\nu_B(\nu_B(f^2) f^1 b)1_Q\\
&=&\nu_B(b)1_Q\\
\end{eqnarray*}
Hence $\bar \rho^{AB}$ is a comeasuring of Frobenius algebras. 
By universal property we have an algebra homomorphism $S_{AB}:\Aa_{AB} \to \Aa_{BA}^{op}=Q$, such that \[
\xymatrix{ B \ar[r]^-{\bar \rho^{AB}} \ar[dr] _{\rho^{AB}} & A\ot Q  \\
           &  A \ot \Aa_{AB}\ar@{.>}[u]_{\exists! A \ot S_{AB}}} 
\]
commutes. So we show these morphisms satisfy the antipode condition of Hopf opcategories. Denote the algebra structure on $\Aa_{BA}$ by $\mu_{BA}$ and $\eta_{BA}$. Let $b\in B$ and calculate
\begin{eqnarray*} \notag
(B\ot \mu_{BA})(\rho^{BA} \ot \Aa_{BA})\bar \rho^{AB}(b)=\nu_B(e^{2[0]} b)e^{1[0]} \ot e^{1[1]} e^{2[1]}=\nu_B(f^2 b) f^{1}\ot \eta_{AB}=\Id_B(b) \ot \eta_{AB}.
\end{eqnarray*}
This shows, that:
\begin{eqnarray*}
&& (B\ot \mu_{BA})(B \ot \Aa_{BA}\ot  S_{AB})(B \ot d_{BAB}) \rho^{BB}\\
&=&(B\ot \mu_{BA})(B \ot \Aa_{BA}\ot S_{AB})(\rho^{BA} \ot \Aa_{AB}) \rho^{AB}\\
&=&(B\ot \mu_{BA})(\rho^{BA} \ot \Aa_{BA})\bar \rho^{AB}\\
&=&(B\ot \eta_{AB}e_{B})\rho^{BB}
\end{eqnarray*}
Hence by universal property of $\Aa_{BB}$ we obtain $\mu_{BA}(\Aa_{BA}\ot S_{AB})d_{BAB}=\eta_{BA}e_{B}.$ Further
\begin{eqnarray*}
(A\ot \mu_{BA})(\bar \rho^{AB} \ot \Aa_{BA})\rho^{BA}(a)&=&e^1 \nu_B(e^{2[0]}a_{(0)})\ot e^{2[1]}a_{(1)}= \nu_B((e^2a)^{[0]})e^1\ot (e^2a)^{[1]}\\
&=&e^1 \ot \nu_A(e^2a)\eta_{BA}=\Id_{A}(a) \ot \eta_{BA}, 
\end{eqnarray*}
which proves
\begin{eqnarray*}
&&(A\ot \mu_{BA})(A\ot S_{AB}\ot \Aa_{BA})(A\ot d_{ABA}) \rho^{AA} \\
&=& (A\ot \mu_{BA})(A\ot S_{AB}\ot \Aa_{BA})(\rho^{AB} \ot \Aa_{BA}) \rho^{BA}\\
&=&(A\ot \mu_{BA})(\bar \rho^{AB} \ot \Aa_{BA})\rho^{BA}\\
&=&(A\ot \eta_{BA}e_{A})\rho^{AA}.
\end{eqnarray*}
Therefore by universal property of $\Aa_{AA}$ we have $\mu_{BA}(S_{AB}\ot \Aa_{BA})d_{ABA}=\eta_{AB}e_{A}.$\\
To show that $S_{AB}$ is invertible, denote $Q:=\Aa_{BA}^{op}$ and define $\rho'^{AB}: B \to A\ot Q$ by 
$$ \rho'^{AB}(b):=e^2 \nu_B(b e^{1[0]}) \ot e^{1[1]},$$
where $\rho^{BA}(e^1)=e^{1[0]}\ot e^{1[1]}$. We show $\rho'$ is a comeasuring from $B$ to $A$. Let therefore $b,b'\in B$ and calculate
\begin{eqnarray*}
\rho'^{AB} (b) \rho'^{AB}(b')&=&e^2 \nu_B(b e^{1[0]})E^2 \nu_B(b' E^{1[0]}) \ot E^{1[1]} e^{1[1]}\\
&=&e^2 \nu_B(b (E^2e^1)^{[0]}) \nu_B(b' E^{1[0]}) \ot E^{1[1]} (E^2e^1)^{[1]}\\
&=&e^2 \nu_B(b E^{2[0]}e^{1[0]}) \nu_B(b' E^{1[0]}) \ot E^{1[1]} E^{2[1]}e^{1[1]}\\
&=& e^2 \nu_B(b f^2e^{1[0]}) \nu_B(b' f^1) \ot e^{1[1]} \\
&=&e^2 \nu_B(b f^2b' e^{1[0]}) \nu_B(f^1) \ot e^{1[1]}\\
&=&e^2 \nu_B(bb' e^{1[0]})  \ot e^{1[1]}\\
&=& \rho'^{AB}(bb'),
\end{eqnarray*}
where we build the product in $A\ot \Aa_{BA}^{op}$ in line 1, we use the Casimir property of $e^1\ot e^2$ in line 2, the first equation of \equref{comeas1} in line 3, \equref{comeasfrob2} in line 4, the Casimir property of $f^1\ot f^2$ in line 5 and finally $\nu_B(f^1)f^2=1_B$ in line 6. For the units we have
\begin{eqnarray*}
\rho'^{AB}(1_B)=e^2\ot \nu_B(e^{1[0]}) e^{1[1]}= \nu_B(e^1)e^2\ot  1_{Q}= 1_A\ot 1_{Q},
\end{eqnarray*}
using \equref{comeas2} for $a=e^1$. We show that $\rho'^{AB}$ satisfies \equref{comeas1} for $b\in B$ by 
\begin{eqnarray*}
&&(A\ot A\ot \mu_{Q})(A\ot \sigma \ot Q)(\rho'^{AB}\ot \rho'^{AB})d_B(b)\\
&=&e^2\nu_B(f^1 e^{1[0]})\ot E^2 \nu_B(f^2 b E^{1[0]}) \ot E^{1[1]} e^{1[1]}\\
&=&e^2\nu_B(b E^{1[0]} f^1 e^{1[0]})\ot E^2 \nu_B(f^2 ) \ot E^{1[1]} e^{1[1]}\\
&=&e^2\nu_B(b E^{1[0]} e^{1[0]}) \ot E^2 \ot E^{1[1]} e^{1[1]}\\
&=&e^2\nu_B(b(E^1e^1)^{[0]})\ot E^2  \ot (E^1e^1)^{[1]}\\
&=& \nu_B(be^{1[0]}) e^2E^1 \ot E^2 \ot e^{1[1]}\\
&=&(\Delta_A\ot Q) \rho'^{AB}(b),
\end{eqnarray*}
where we used that multiplication $\mu_Q$ in $Q$ is the opposite of the multiplication in $\Aa_{BA}$ and the characterisation of the comultiplication in $B$ by the Casimir element $f^1\ot f^2$ in line 2, the Casimir property for $f^1\ot f^2$ in line 3, $\nu_B(f^2)f^1=1_B$ in line 4, the first equation in \equref{comeas1} in line 5, Casimir property for $e^1\ot e^2$ in line 6 and finally the characterisation of the comultiplication in $A$ by the Casimir element $E^1\ot E^2$ in line 7. For the counits we show 
\begin{eqnarray*}
(\nu_A\ot Q) \rho'^{AB} (b)&=& \nu_B(b e^{1[0]}) e^{1[1]}\nu_A(e^2)=\nu_B(b e^{1[0]}) \nu_A(e^{2[0]})e^{1[1]}e^{2[1]}\\
&=&\nu_B(b f^1)\nu_B(f^2)1_{Q}=\nu_B(b)1_{Q},
\end{eqnarray*}
using \equref{comeas2} for $a=e^2$ and \equref{comeasfrob2}. Hence, $\rho'^{AB}$ is a comeasuring of Frobenius algebras and therefore by universal property there exists an algebra homomorphism $S'_{AB}: \Aa_{AB}\to \Aa_{BA}^{op}$, such that 
\[
\xymatrix{ B \ar[r]^-{\rho'^{AB}} \ar[dr] _{\rho^{AB}} & A\ot Q  \\
           &  A \ot \Aa_{AB}\ar@{.>}[u]_{\exists! A \ot S'_{AB}}
} 
\]
commutes.
Finally $S'_{AB}$ is inverse to $S_{BA}$, since for all $b\in B$ 
\begin{eqnarray*}
(A \ot S_{BA}S'_{AB})\rho^{AB}(b)&=&(A \ot S_{BA}) \rho'^{AB}(b)\\
&=&(A \ot S_{BA})(e^2 \nu_B(be^{1[0]} \ot e^{1[1]})\\
&=&e^2 \nu_B(be^{1[0]}) \ot S_{BA}(e^{1[1]})\\
&=&e^2 \nu_B(b\nu_A(f^{2[0]}e^1) f^1)) \ot f^{2[1]}\\
&=&\nu_A(e^1) e^2 f^{2[0]}\nu_B(b f^1) \ot f^{2[1]}\\
&=&\nu_B(bf^1) \rho^{AB}(f^2)=\nu_B(b^{(1)}) \rho^{AB}(b^{(2)})\\
&=&\rho^{AB}(b),
\end{eqnarray*}
where we used $(A\ot S'_{AB})\rho^{AB}= \rho'^{AB}$ in line 1 and $(A\ot S_{BA})\rho^{BA}= \bar \rho^{BA}$ in line 4, furthermore the Casimir property for $e^1\ot e^2$ in line 5, $\nu_A(e^1)e^2=1_A$ and the characterisation of the comultiplication in $B$ by the Casimir element $f^1\ot f^2$ in line 6 and finally counitality in $B$.
Conversely, for all $a\in A$ 
\begin{eqnarray*}
(B\ot S'_{AB}S_{BA})\rho^{BA}(a)&=&(B \ot S'_{AB})\bar \rho^{BA}(a)\\
&=&(B \ot S'_{BA})(f^1 \nu_A(f^{2[0]}a) \ot f^{2[1]}\\
&=&f^1 \nu_A(f^{2[0]}a) \ot S'_{BA}(f^{2[1]})\\
&=&f^1 \nu_A(\nu_B(f^2 e^{1[0]}) e^2a) \ot e^{1[1]}\\
&=&\nu_B(f^2) e^{1[0]} f^1 \nu_A(e^2a) \ot e^{1[1]}\\
&=&\nu_A(e^2a) \rho^{BA}(e^1)=\nu_A(a^{(2)}) \rho^{BA}(a^{(1)})\\
&=&\rho^{BA}(a),
\end{eqnarray*}
where we used $(B\ot S'_{BA})\rho^{AB}= \rho'^{AB}$ and $(A\ot S_{BA})\rho^{BA}= \bar \rho^{BA}$, furthermore the Casimir property for $f^1\ot f^2$ in line 5, $\nu_A(e^1)e^2=1_A$ in line 6. Hence by universal property of $\Aa_{AB}$ we find that $S_{BA}S'_{AB}=\Aa_{AB}$ and similarly $S'_{AB}S_{BA}=\Aa_{BA}$. Hence $S$ is bijective. 
\end{proof}

\section{Dualities}\label{se:dualities}

\subsection{Sweedler dual of a semi-Hopf opcategory}

Let $A$ be an algebra. Recall that the Sweedler dual $A^\circ$ of $A$ is the set of all linear functionals $f\in A^*$ such that $\ker f$ contains a two-sided ideal of finite codimension. Then $A^\circ$ is a coalgebra in a natural way, where
$$\Delta(f)=f^{(1)}\ot f^{(2)}$$ if and only if 
$$f(ab)=f^{(1)}(a)f^{(2)}(b)$$
for all $a,b\in A$. Moreover $(-)^\circ :\Alg \to \Coalg$ is a contravariant functor which is adjoint to the contravariant functor $(-)^*:\Coalg\to \Alg$, that sends a coalgebra $C$ to the convolution algebra $C^*$.
Then we have the following immediate result.

\begin{lemma}\label{le:dualHcat}
Let $\ul A$ be a semi-Hopf opcategory. Then there is a semi-Hopf category $\ul A^\circ$ defined as follows: objects of $\ul A^\circ$ are the same as objects of $\ul A$, and for each pair of objects $x,y \in A^0 = (A^\circ)^0$, we have 
$$(A^\circ)_{xy}=(A_{xy})^\circ$$
Moreover, if $\ul A$ is a Hopf opcategory, then $\ul A^\circ$ is a Hopf category.
\end{lemma}

\begin{proof}
We already know that $A^\circ_{xy}$ is a coalgebra for each pair of objects $x,y$. Let us show that it also is a $k$-linear category. To this end, take any triple of objects $x,y,z$ and define 
$$m_{xyz}: A^\circ_{xy}\ot A^\circ_{yz}\to A^\circ_{xz}$$ 
as follows. For $f\in A^\circ_{xy}$ and $g\in A^\circ_{yz}$, we have a functional $(f\ot g)\circ d_{xyz}\in A^*_{xz}$. Now let $I$ and $J$ be ideals of finite codimension respectively in $\ker f$ and $\ker g$. Let $K\subset A_{xz}$ be the inverse image of $I\ot A_{yz}+ A_{xy}\ot J$ under $d_{xyz}$. Then $K$ is an ideal of finite codimension contained in the kernel of $(f\ot g)\circ d_{xyz}$, hence the latter is in $A^\circ_{xz}$, which shows that $m_{xyz}$ is well-defined by $m_{xyz}(f\ot g)=(f\ot g)\circ d_{xyz}$. Furthermore, $m_{xyz}$ is a coalgebra map as $d_{xyz}$ is an algebra map. Finally, we define $j_{x}(1_k)=e_x$, which clearly is in $A^\circ_{xx}$ being an algebra map onto $k$.

If $S$ is an antipode for $\ul A$ then their dual morphisms form an antipode for $\ul A^\circ$.
\end{proof}

\begin{remark}
The construction of Lemma~\ref{le:dualHcat} is not functorial. Indeed, if $\ul f:\ul A\to \ul B$ is a morphism of semi-Hopf opcategories, then $\ul f$ sends objects of $A$ to objects of $B$, and for each pair of objects $x,y\in A^\circ$, we have that $f_{xy}:A_{xy}\to B_{fxfy}$ is an algebra morphism (satisfying coherence). When we turn to Sweedler duals, we find that each $f_{xy}$ induces a coalgebra morphism $f^*_{xy}: B_{fxfy}^\circ \to A_{xy}^\circ$, however on the level of objects, we still go from objects in $\ul A$ (or $\ul A^\circ$) to objects in $\ul B$ (or $\ul B^\circ$).
\end{remark}

We will now combine the above result with \cite[Thm. 3.20, Thm. 4.14]{AGV}, which we state here in the finite dimensional case (which is sufficient for our needs, since Frobenius algebras are finite dimensional) and reformulate in terms of semi-Hopf categories.

\begin{theorem} \label{meascomeas}
\begin{enumerate}
\item
Let $A$ and $B$ be finite dimensional $\Omega$-algebras and $(\Aa_{BA},\rho^{BA})$ their universal comeasring algebra. Let $\tilde \psi_{AB}: \Aa^{\circ}_{BA} \ot A \to B$ the measuring defined by 
$$\tilde{\psi}_{BA}(f \ot a):=(B \ot f) \rho^{BA}(a)=f\left(a_{(1)}\right) a_{(0)},\qquad  \forall f \in \Aa_{BA}^{\circ}, a \in A.$$ Then there exist a coalgebra isomorphism $\theta_{BA}: \Aa_{BA}^{\circ} \to \Cc_{BA}$ such that
\[
\xymatrix{\Cc_{BA} \ot A \ar[r]^-{\psi_{BA}} & B\\
\Aa_{BA}^\circ \ot A \ar@{..>}[u]^{\exists \theta_{BA}\ot A} \ar[ur]_{\tilde \psi_{BA}} & 
}
\]
commutes. 
\item
If $\Xx$ a class of finite dimensional $\Omega$-algebras, $\Cc(\Xx)$ the semi-Hopf category of universal measuring coalgebras as constructed in Proposition~\ref{prop:sHcatOmega} and $\Aa(\Xx)$ the semi-Hopf opcategory of universal comeasuring algebras as constructed in Proposition~\ref{prop:sHopcatOmega}, then we have an isomorphism of semi-Hopf opcategories
$$\theta:\xymatrix{ \Aa(\Xx)^\circ \ar[rr]^-\cong  && \Cc(\Xx)}$$
where $\Aa(\Xx)^\circ$ is the Sweedler dual of $\Aa(\Xx)$ as in Lemma~\ref{le:dualHcat}, being the identity on objects. 
\item 
If $\Xx$ is a class of Frobenius algebras, then the isomorphism $\theta:\xymatrix{ \Aa(\Xx)^\circ \ar[r]^-\cong & \Cc(\Xx)}$ as in part (2) is an isomorphism of Hopf categories. In particular, $\theta_{AA}: \Aa_{AA}^{\circ} \to \Cc_{AA}$ is a Hopf algebra isomorphism. 
\end{enumerate}
\end{theorem}

\begin{proof}
Part (1) follows directly from \cite[Thm. 3.20., Thm. 4.14]{AGV}. 

By Lemma~\ref{le:dualHcat}, we know that $\Aa(\Xx)^{\circ}$ is indeed a semi-Hopf category, and Part (1) defines a family of coalgebra morphisms $\theta_{AB}$. 
We prove now the compatibility with the multiplicative structure. Take $A,B,C$ in $\Xx$, $f \in \Aa_{AB}^\circ$, $g\in \Aa_{BC}^\circ$ and $c\in C$. Then we find 
\begin{eqnarray*}
\psi_{AC}(\theta_{AC} m^{\Aa^\circ}_{ABC}  \ot C)(f\ot g\ot c)&=&\tilde \psi_{AC} (m^{\Aa^\circ}_{ABC} (f\ot g) \ot c)\\
&=&(m^{\Aa^\circ}_{ABC} (f\ot g)\ot A)\rho^{AC}(c)\\
&=&(f\ot g \ot A)(d_{ABC} \ot A)\rho^{AC} (c)
\end{eqnarray*}
and
\begin{eqnarray*}
\psi_{AC}(m^\Cc_{ABC}\ot C) (\theta_{AB}\ot \theta_{BC}\ot C)(f\ot g\ot c)
&=&\psi_{AB}(\Cc_{AB} \ot \psi_{BC})(\theta_{AB}(f)\ot \theta_{BC}(g)\ot c)\\
&=& \tilde \psi_{AB} (\Aa_{AB}^\circ \ot \tilde \psi_{BC})(f\ot g\ot c)\\ 
&=&(f\ot g \ot A)(\rho^ {AB} \ot \Aa_{BC})\rho^{BC} (c).
\end{eqnarray*}
By definition of the opcategory structure of $\Aa$, we have that 
$$(\rho^ {AB} \ot \Aa_{BC})\rho^{BC}=(d_{ABC} \ot A)\rho^{AC}$$
and therefore the above computations imply that
$$\psi_{AC}(\theta_{AC} m^{\Aa^\circ}_{ABC}  \ot \Cc)=\psi_{AC}(m^\Cc_{ABC}\ot C) (\theta_{AB}\ot \theta_{BC}\ot C).$$
The universal property of the comeasuring $(\Cc_{AC},\psi_{AC})$ (see Lemma~\ref{uniquemorphismuniversal}) then implies that indeed
$$\theta_{AC} m^{\Aa^\circ}_{ABC}=m^\Cc_{ABC}\circ  (\theta_{AB}\ot \theta_{BC}).$$
In a similar way, one verifies that $\theta$ is compatible with the units. Indeed, by definition of the $\Vv$-category structure of $\Aa^\circ$ and the opcategory structure of $\Aa$ we find 
$$\psi_{AA}(\theta_{AA} j_A^{\Aa^\circ} \ot A)= \tilde \psi_{AA}(j_A^{\Aa^\circ} \ot A) )= \tilde\psi_{AA}(e_A\ot A) = (A\ot e_A)\circ \rho^{AA}=A.$$
On the other hand, the definition of the $\Vv$-category structure of $\Cc$ guarantees that
$$\psi_{AA}(j^\Cc_A\ot A)= A$$
Again by Lemma~\ref{uniquemorphismuniversal} we then find that $\theta_{AA} j_A^{\Aa^\circ}=j^\Cc_A$.

Part (3) follows directly from part (2) and the fact that a semi-Hopf category morphism between Hopf categories is automatically preserving the antipodes.
\end{proof}

\subsection{Duality for Frobenius algebras and the antipode of their universal Hopf category}\selabel{duality2}

Recall that a Frobenius algebra $A$ is always finite dimensional and a dual base is given by $e^1\otimes \nu(e^2-)\in A\otimes A^*$ (where as usual, we denote by $e^1\ot e^2$ the Casimir element of $A$ and by $\nu$ its faithful functional or counit).  
Consequently, $A^*$ is again a Frobenius algebra whose structure is given by 
\begin{eqnarray*}
(f\cdot g)(a):=f(a^{(2)})g(a^{(1)})=f(e^2a)g(e^1)=f(e^2)g(ae^1), & 1_{A^*}:=\nu_A,\\
\Delta (f)(a\ot b):= f^{(1)}(a)f^{(2)}(b)=f(ba), & \nu_{A^*}(f):= f(1),
\end{eqnarray*}
for all $f,g \in A^*, a,b\in A$. Remark that the Casimir element of $A^*$ is given by
$$\nu(e^1-)\ot \nu(e^2-)\in A^*\ot A^*.$$
Moreover, the maps $\iota: A\to A^*, \iota(a)(b)=\nu(ab))$ and $\iota^{-1}: A^*\to A, f\mapsto f(e^1)e^2$ are mutual inverse isomorphisms of Frobenius algebras.

\begin{lemma}\label{MeasIso}
Let $A,B$ be Frobenius algebras, with duals $A^*,B^*$. Then $\Meas(A,B)$ and $\Meas(A^*,B^*)$ are isomorphic as categories. Dually $\Comeas(A,B)$ and $\Comeas(A^*,B^*)$ are isomorphic as categories.
\end{lemma}

\begin{proof}
Given a measuring $\psi: P\ot A \to B$ we can construct a measuring $\psi^*: P \ot  A^*\to B^*$ by 
$$\psi^{*}(p\ot f):= \iota_B(\psi(p \ot \iota_A^{-1}(f))).$$
Since this is a composition of three coalgebra morphisms, it is a coalgebra morphism itself. Moreover, $\psi^*$ is a measuring of algebras since
\begin{eqnarray*}
\psi^*(p\ot fg)&=&\iota_B(\psi(p\ot \iota^{-1}_A(f)\iota^{-1}_A(g))=\iota_B(p^{(1)}(\iota^{-1}_A(f))p^{(2)}(\iota^{-1}_A(g)))\\
&=&\iota_B(p^{(1)}(\iota^{-1}_A(f)))\iota_B(p^{(2)}(\iota^{-1}_A(g)))=\psi^*(p^{(1)}\ot f)\psi^*(p^{(2)}\ot g)
\end{eqnarray*}
and 
$$\psi^{*}(p\ot1_{A^*})=\iota_B(\psi(p \ot 1_A))=\epsilon(p)\iota_B(1_B) =\epsilon(p) 1_{B^*}.$$
Furthermore, let $\psi': P'\ot A\to B$ be another measuring and $\phi:P\to P'$ a coalgebra morphism such that $\psi=\psi'(\phi\ot A)$ then 
$$\psi'^{*}(\phi \ot A)= \iota_B \psi' (\phi\ot \iota_A^{-1})=\iota_B \psi (P\ot \iota_A^{-1})=\psi^*.$$ 
The assignment $(P,\psi)\mapsto (P,\psi^*)$ and $\phi\mapsto \phi$ forms a functor $\Meas(A,B)\to \Meas(A^*,B^*)$. 

Conversely, given a measuring $\psi^*: P \ot  A^*\to B^*$ we can construct a measuring $\psi: P\ot A\to B$ by $\psi(p\ot a):=\iota^{-1}_B(\psi^*(p\ot \iota_A(a)))$. We can again lift this to a functor $\Meas(A^*,B^*) \to \Meas(A,B)$ by setting it to be the identity on morphisms. These two functors are inverse to each other, hence $\Meas(A,B)$ and $\Meas(A^*,B^*)$ are isomorphic as categories. 

Similarly any comeasuring $\rho:A\to B\ot Q$ gives rise to a comeasuring $\rho^*:=(\iota_B\ot Q)\rho\iota^{-1}_A: A^*\to B^* \ot Q$ and the assignment is bijective and functorial.
\end{proof}

Hence we obtain the following theorem.

\begin{proposition}
\begin{enumerate}[(1)]
\item
Let $A,B$ be Frobenius algebras and $A^*,B^*$ their dual Frobenius algebras. Then there is a unique coalgebra isomorphism
$$\gamma_{AB}: \xymatrix{\Cc_{AB} \ar[rr]^-\cong && \Cc_{A^*B^*}},$$
rendering the following diagram commutative 
\[
\xymatrix{
\Cc_{AB}\ot B \ar[d]_-{\psi_{AB}} \ar[rr]^-{\Cc_{AB}\ot\iota_B} && \Cc_{AB}\ot B^* \ar[rr]^-{\gamma_{AB}\ot B^*} 
&& \Cc_{A^*B^*}\ot B^* \ar[d]^-{\psi_{A^*B^*}} \\
A \ar[rrrr]^-{\iota_A} &&&& A^*.
}
\]
\item
Let then $\Xx$ be any class of Frobenius algebras and denote by $\Xx^*$ the class their dual Frobenius algebras. Then the isomorphisms $\gamma_{AB}$ as defined in part (1), for all $A,B\in \Xx$, form an isomorphism of Hopf categories
$$\gamma:\Cc(\Xx)\to \Cc(\Xx^*),$$
sending objects $A$ to their duals $A^*$.
In particular,
$$\gamma_{AA}: \Cc_{AA} \simeq \Cc_{A^*A^*}$$
is a Hopf algebra isomorphism for any Frobenius algebra $A$.
\end{enumerate}
\end{proposition}

\begin{proof}
Since the categories $\Meas(A,B)$ and $\Meas(A^*,B^*)$ are isomorphic by Lemma~\ref{MeasIso}, so are their final objects. This induces the coalgebra isomorphism $\gamma_{AB}$ as in part (1). For part (2), we only show the compatibility with the multiplication and unit. We can compute
 \begin{eqnarray*}
&&\psi_{A^*C^*}( m_{A^*B^*C^*} \ot C^*)(\gamma_{AB}\ot \gamma_{BC}\ot C^*)\\
&=&\psi_{A^*B^*} (\Cc_{A^*B^*}\ot \psi_{B^*C^*}) (\gamma_{AB}\ot \gamma_{BC}\\
&=&\psi^*_{A^*B^*} (\Cc_{AB}\ot \psi^*_{B^*C^*})\\
&=&\iota_A\psi_{AB} (\Cc_{AB}\ot \iota^{-1}_B\iota_B)(\Cc_{AB}\ot \psi_{BC})(\Cc_{AB}\ot \Cc_{BC}\ot \iota^{-1}_C)\\
&=&\iota_A\psi_{AC} (m_{ABC}\ot C)(\Cc_{AB}\ot \Cc_{BC}\ot \iota^{-1}_C)\\
&=&\psi^*_{A^*C^*}(m_{ABC} \ot C^*)=\psi_{A^*C^*}(\gamma_{AC} m_{ABC} \ot C^*),
\end{eqnarray*}
and
 \begin{eqnarray*}
\psi_{A^*A^*}(\gamma_{AA}j_A\ot A^*)=\psi^*_{A^*A^*}(j_A\ot A^*)=\iota_A \psi_{AA} (j_{A}\ot \iota_A^{-1})= \iota_A\iota_A^{-1}=A^*=\psi_{A^*A^*}(j_{A^*}\ot A^*)
\end{eqnarray*}
 by universal property of $\Cc_{A^*C^*}$ and $\Cc_{A^*A^*}$ we find that $\gamma$ is a morphism of semi-Hopf categories.
\end{proof}

Recall from \cite[Thm. 4.16.]{AGV} the existence of an ismorphism between the opposite universal acting bialgebra of any $\Omega$-algebra and the universal acting bialgebra of its dual. One can show a similar duality for general measurings. We formulate this result here in the case of Frobenius algebras. 

\begin{proposition}\label{isomorphismpi}
Let $A,B$ be Frobenius algebras and $A^*, B^*$ their duals. Denote their universal measuring by $\psi_{B^*A^*}: \Cc_{B^*A^*} \ot A^* \to B^*$. Define $\hat \psi_{AB}: \Cc_{B^*A^*}^{cop} \ot B\to A$ by 
$$\hat \psi_{AB}(p \ot b):= \sum_i \psi_{B^*A^*}(p\ot f_i)(b) e_i,\qquad \forall p \in \Cc_{B^*A^*}, b \in B,$$
where $\{(e_i,f_i)\}\subset A\times A^*$ is a finite dual base for $A$. 
Then this induces a unique coalgebra isomorphism $\pi_{AB}: \Cc_{B^*A^*}\to \Cc_{AB}^{cop} $ such that
\begin{equation}\eqlabel{equationduality}
\xymatrix{ \Cc_{B^*A^*} \ot B \ar[r]^-{\hat \psi_{AB}} \ar@{.>}[d]_{\pi_{AB} \ot B} & A \\
            \Cc_{AB} \ot B \ar[ur] _{\psi_{AB}} & }
\end{equation}            
commutes. Furthermore, for Frobenius algebras $A,B,C$ with duals $A^*,B^*,C^*$ we have
$$\pi_{CA}m_{A^*B^*C^*}=m_{CBA}\sigma(\pi_{BA}\ot \pi_{CB})$$
and $\pi_{AA}j_{A^*}=j_A.$

In other words, if $\Xx$ is a class of Frobenius algebras, then there is an isomorphism of Hopf categories $\pi:\Cc(\Xx^*)\to \Cc(\Xx)^{op,cop}$ defined component-wise as above.
In particular,
$$\pi_{AA}: \Cc_{A^*A^*}\simeq \Cc_{A,A}^{cop,op}$$
is a Hopf algebra isomorphism. 
\end{proposition}

\begin{proof}
To check the conditions of a measuring, let $b,c\in B, p\in \Cc_{B^*A^*}$ and $f\in A^*$. Then we find
\begin{eqnarray*}
f(\hat \psi_{AB}(p \ot bc))&=&\psi_{B^*A^*}(p\ot f)(bc)=\psi_{B^*A^*}(p \ot f)^{(1)}(c)\psi_{B^*A^*}(p\ot f)^{(2)}(b)\\
&=&\psi_{B^*A^*}(p^{(1)} \ot f^{(1)})(c)\psi_{B^*A^*}(p^{(2)}\ot f^{(2)})(b)\\
&=&f^{(1)}(\hat \psi_{AB}(p^{(1)} \ot c)) f^{(2)}(\hat \psi_{AB}(p^{(2)}\ot b))\\
&=&f(\hat \psi_{AB}(p^{(2)}\ot b) \hat \psi_{AB}(p^{(1)} \ot c)),
\end{eqnarray*}
where we used the definition of the comultiplication in $B^*$ in line 1, \equref{meascoalg} for $\psi_{B^*A^*}$ in line 2 and the definition of the comultiplication in $A^*$ in line 4. Moreover, 
\begin{eqnarray*}
f(\hat \psi_{AB}(p \ot 1_B))&=&\psi_{B^*A^*}(p\ot f)(1_B)=\nu_{B^*}(\psi_{B^*A^*}(p\ot f))\\
&=&\epsilon_{B^*A^*}(p)\nu_{A^*}(f)=\epsilon_{B^*A^*}(p)f(1_A),
\end{eqnarray*}
where we used the definition of the counit in $B^*$ in line 1, \equref{meascoalg} for $\psi_{B^*A^*}$ in line 2 and finally the definition of the counit in $A^*$. 
Since these equations hold for arbitrary $f\in A^*$, we find that $\hat \psi_{AB}$ is a measuring of algebras. To show that it is also a measuring of coalgebras let $b\in B, p\in \Cc(A^*,B^*), f,g\in A^*$. Then:
\begin{eqnarray*}
&&(f\ot g) \Delta_A (\hat \psi_{AB}(p \ot b))\\
&=&f(\hat \psi_{AB}(p \ot b)^{(1)})g(\hat \psi_{AB}(p \ot b)^{(2)})\\
&=&(gf)(\hat \psi_{AB}(p \ot b))=\psi_{B^*A^*}(p \ot gf)(b)\\
&=&(\psi_{B^*A^*}(p^{(1)} \ot g)\psi_{B^*A^*}(p^{(2)} \ot f))(b)\\
&=&(\psi_{B^*A^*}(p^{(2)} \ot f))(b^{(1)})(\psi_{B^*A^*}(p^{(1)} \ot g))(b^{(2)})\\
&=&f(\hat \psi_{AB}(p^{(2)} \ot b^{(1)}))g(\hat \psi_{AB}(p^{(1)} \ot b^{(2)}))\\
&=&(f\ot g)(\hat \psi_{AB}(p^{(2)} \ot b^{(1)})\ot \hat \psi_{AB}(p^{(1)} \ot b^{(2)})),
\end{eqnarray*}
where we used the definition of the multiplication in $A^*$ in line 2, \equref{measalg} for $\psi_{B^*A^*}$ in line 3 and the definition of the multiplication in $B^*$ in line 4. Again this holds for arbitrary $f,g\in A^*$ hence $\hat \psi_{AB}$ satisfies the first equation in \equref{meascoalg} and for the counit we have:
\begin{eqnarray*}
\nu_{A}(\hat \psi_{AB}(p \ot b))=\psi_{B^*A^*}(p \ot \nu_{A})(b)=\epsilon_{B^*A^*}(p)\nu_{B}(b),
\end{eqnarray*}
where we used \equref{measalg} for $\psi_{B^*A^*}$ and the fact that $\nu_A$ and $\nu_B$ are the units in $A^*$ and $B^*$ respectively. 
This shows that $\hat \psi_{AB}$ is a measuring of Frobenius algebras from $B$ to $A$ and hence there exists a unique coalgebra morphism $\pi_{AB}: \Cc_{B^*A^*}\to \Cc_{AB}^{cop}$, such that diagram \equref{equationduality} commutes. Now let $A,B,C$ be Frobenius algebras with duals $A^*,B^*,C^*$ and $a\in A, p\in \Cc_{A^*B^*}, q\in \Cc_{B^*C^*}, f\in C^*$.
Then:
\begin{eqnarray*}
f(\psi_{CA} (\pi_{CA} m_{A^*B^*C^*}\ot A)(p\ot q\ot a))&=&f(\hat \psi_{CA}(pq \ot a))\\
&=&\psi_{A^*C^*}(pq\ot f)(a)=\psi_{A^*B^*}(p \ot \psi_{B^*C^*}(q\ot f))(a)\\
&=& \psi_{B^*C^*}(q\ot f)(\hat\psi_{BA}(p \ot a))=f( \hat \psi_{CB}(q\ot \hat\psi_{BA}(p \ot a)))\\
&=&f(\psi_{CB}(\pi_{CB}(q)\ot \psi_{BA}(\pi_{BA}(p) \ot a)))\\
&=&f(\psi_{CA}(m_{CBA}\ot A) (\pi_{CB}(q)\ot \pi_{BA}(p) \ot a)\\
&=&f( \psi_{CA}(m_{CBA}\sigma \ot C)(\pi_{BA} \ot \pi_{CB} \ot C)(p\ot q \ot a)),\\
\end{eqnarray*}
where we used $\psi_{CA}(\pi_{CA}\ot A)=\hat \psi_{CA}$ in line 1, the characterisation of $m_{A^*B^*C^*}$ in line 2, the definition of $\hat \psi_{BA}$ and $\hat \psi_{CB}$ in line 3, $\psi_{CB}(\pi_{CB}\ot B)=\hat \psi_{CB}$ and $\psi_{BA}(\pi_{BA}\ot A)=\hat \psi_{BA}$ in line 4 and the characterisation of $m_{CBA}$ in line 5. 
Finally, for $a\in A$ and $f\in A^*$
$$f(\psi_{AA} (\pi_{AA}j_{A^*}\ot a))= f(\hat \psi_{AA}(j_{A^*}\ot a))=\psi_{A^*A^*}(j_{A^*} \ot f)(a)=f(a),$$
where we recall $\psi_{A^*A^*}(j_{A^*} \ot A^*)=A^*$. Since $f$ was arbitrary and by universal property we obtain the required compatibility with the multiplication and unit. We omit the proof of the fact that $\pi_{AB}$ is an isomorphism, as this follows from Theorem \ref{factor} below.
\end{proof}

The previously obtained dualities $\gamma$ and $\pi$ are related with the antipode as in the statement of the following proposition.

 \begin{theorem} \label{factor}
Let $\Xx$ be a class of Frobenius algebras and $\Xx^*$ the class of their dual Frobenius algebras. We have the following commutative diagram of Hopf category isomorphisms:
\[
\xymatrix{
\Cc(\Xx^*) \ar[rr]^{\pi} & & \Cc(\Xx)^{op,cop}\\
& \Cc(\Xx) \ar[ul]^{\gamma}\ar[ur]_{S}&
}
\]
\end{theorem}

\begin{proof}
Since $\pi\gamma(A)=A$ we know that the morphism agree on the index sets $X$. We show $\pi_{AB}\gamma_{BA}=S_{BA}$. Therefore, let $p\in \Cc_{AB}, b\in B, f\in A^*$ and compute
\begin{eqnarray*}
f(\psi_{AB}(\pi_{AB}\gamma_{BA} \ot B)(p\ot b))&=&f(\hat \psi_{AB}(\gamma_{AB}(p) \ot b))\\
&=&\psi_{B^*A^*}(\gamma_{BA}(p)\ot f) (b)\\
&=&\psi^*_{B^*A^*}(p\ot f) (b)=\iota_B(\psi_{BA} (p\ot \iota^{-1}_A(f))(b)\\
&=&\nu_B(\psi_{BA} (p\ot f(e^1)e^2 )b)=\nu_B(f(e^1)p(e^2)b)\\
&=&f(e^1\nu_B(p(e^2)b))=f(\psi_{AB} (S_{BA}(p)\ot b))\\
&=&f(\psi_{AB}(S_{BA} \ot B)(p\ot b)).
\end{eqnarray*}
Since $f$ was arbitrary and by universal property of $\Cc_{AB}$ we have $\pi_{AB}\gamma_{BA}=S_{BA}$. Since $S$ and $\gamma$ are isomorphisms, it follows that $\pi$ is an isomorphism as well.
\end{proof}

The results of this section can be dualized for comeasurings. We will avoid explicit proofs and just state the concluding result in this case.

\begin{theorem} 
Let $\Xx$ be a class of Frobenius algebras and $\Xx^*$ the class of the dual of algebras in $\Xx$. We have a commutative diagram of isomorphisms of Hopf opcategories as follows:
\[
\xymatrix{
\Aa(\Xx^*) \ar[rr]^{\varpi} & &\Aa(\Xx)^{op,cop}\\
& \Aa(\Xx) \ar[ul]^{\alpha}\ar[ur]_{S}&
}
\]
\end{theorem}

\begin{remark}
Let us illuminate how to define the isomorphisms $\alpha$ and $\varpi$ dual to the above morphisms $\gamma$ and $\pi$ for universal measuring coalgebras. The isomorphism $\alpha$ can be constructed by assigning to every comeasuring $\rho:A\to B\ot Q$ a comeasuring $\rho^*:=(\iota_B \ot Q)\rho \iota^{-1}_A: A^*\to B^*\ot Q$, where $\iota_A:A\to A^*$ and $\iota_B:B\to B^*$ are the isomorphisms of Frobenius algebras recalled at the beginning of this section.
Furthermore, given the universal comeasuring $\rho^{B^*A^*}$ from $A^*$ to $B^*$, let $\hat \rho^{AB}: B\to \Aa_{B^*A^*}^{op} \ot B$ be the comeasuring defined by 
$$((f\ot \Aa)\hat \rho^{AB})(b):=\rho^{B^*A^*}(f)(b\ot \Aa_{B^*A^*}), \qquad \forall  b \in B, f\in A^*.$$ 
Then this gives rise to the isomorphism $\varpi$. 
\end{remark}

\section{Examples}\selabel{examples}

The notion of a measuring between Frobenius algebras is rather restrictive. Let us illustrate this by considering the trivial Frobenius algebra $k$ (where all structure maps are identities), and considering a non-zero comeasuring $\rho: A\to k\ot Q\cong Q$, where $A$ is an arbitrary Frobenius algebra. By the second axiom of \equref{comeas1} we then already find that
$$\rho(a)=\nu(a)1_Q,$$
so that $\rho$ is uniquely determined. Furthermore, the first axiom of \equref{comeasalg} then implies that
$$\nu(aa')=\nu(a)\nu(a'),$$
which means that the counit of $A$ is multiplicative. However, combining this with the counit property and Frobenius property, we find that 
$$a=\nu(ae^1)e^2=\nu(a)\nu(e^1)e^2=\nu(a)1_A$$
for any $a\in A$. In other words, we find that no non-trivial comeasurings exist from $A$ to $k$, unless $A$ is already isomorphic to $k$. Consequently, we find that $\Aa_{A,k}=\Aa_{k,A}=\Cc_{A,k}=\Cc_{k,A}=0$ whenever $A$ is non-isomorphic to $k$.

In this remaining section we compute some further examples of universal measuring coalgebras and comeasuring algebras. We follow the procedure described in  the proof of \cite[Theorem 3.16]{AGV} to construct these objects. Throughout this section, we suppose that $k$ is a field (and not just a commutative ring), which allows us to explicitly describe algebras and coalgebra structures in terms of their base elements.

Since it is more convenient to work with the tensor algebra and factoring out ideals in the comeasuring case than dealing with subcoalgebras of the cofree coalgebra in the measuring case, we will, given two Frobenius algebras $A$ and $B$, first construct $\Aa_{AB}$ by means of generators and relations and then construct $\Cc_{AB}=\Aa_{AB}^{\circ}$ by duality. By specifying $A$ and $B$ to particular small Frobenius algebras, we will provide an explicit description of the universal measuring coalgebra in these cases. As it will turn out, all universal measuring coalgebras we obtain are finite dimensional.

Another aim is to show that our Hopf category of Frobenius algebras is richer than the ordinary category of Frobenius algebra, meaning there are examples where the universal measuring coalgebra is not generated by grouplike elements (i.e. (iso)morphisms of Frobenius algebras). Furthermore we will show that there are non-isomorphic Frobenius algebras $A$ and $B$ for which $\Cc_{BA}$ is non zero. 

Consider Frobenius algebras $A$ and $B$ and choose a (finite) base $(a_{\alpha})_{\alpha\in I}$ of $A$ and $(b_\beta)_{\beta\in J}$ of $B$. Define the map $\rho_0: A\to B\ot (B^*\ot A)$ given by 
$$\rho_0(a_\alpha)=\sum_{\beta\in J}b_\beta\ot (b_\beta^* \ot a_\alpha),$$
where $b_\beta^*$ are the dual base vectors for the chosen base for $B$.
Now we denote the structure constants $a_{m;\alpha_1,\alpha_2}^\gamma, a_{\Delta;\alpha}^{\gamma_1,\gamma_2},a_{1}^\gamma, a_{\nu;\alpha} \in k$ of $A$ as follows
\begin{eqnarray*}
a_{\alpha_1}a_{\alpha_2}=\sum_{\gamma \in I} a_{m;\alpha_1,\alpha_2}^\gamma a_{\gamma},& 1=\sum_{\gamma \in I} a_{1}^\gamma a_{\gamma},\\\\
\Delta(a_\alpha)=\sum_{\gamma_1,\gamma_2 \in I} a_{\Delta;\alpha}^{\gamma_1,\gamma_2} (a_{\gamma_1}\ot a_{\gamma_2}), & 
\nu(a_\alpha)=a_{\nu;\alpha}.
\end{eqnarray*}
for all $\alpha,\alpha_1,\alpha_2,\gamma,\gamma_1,\gamma_2 \in I$.
Similarly we denote the structure constants $b_{m;\beta_1,\beta_2}^\mu, a_{\Delta;\beta}^{\mu_1,\mu_2},a_{1}^\mu, a_{\nu;\beta} \in k$ of $B$ for $\beta,\beta_1,\beta_2,\mu,\mu_1,\mu_2 \in J$. We now define $\Aa_{BA}:=T\left(B^{*} \otimes A\right)/I$ where $T\left(B^{*} \otimes A\right)$ is the tensor algebra over the vector space $B^{*} \otimes A$ and $I$ is the ideal generated by the elements
\begin{eqnarray*}
\sum_{\gamma\in I} a_{m;\alpha_1,\alpha_2}^\gamma (b_\mu^*\ot a_\gamma)&-&\sum_{\beta_1,\beta_2\in J} b_{m;\beta_1,\beta_2}^\mu (b_{\beta_1}^*\ot a_{\alpha_1})(b_{\beta_2}^*\ot a_{\alpha_2}),\\
\sum_{\gamma_1,\gamma_2 \in I} a_{\Delta;\alpha}^{\gamma_1,\gamma_2} (b_{\mu_1}^*\ot a_{\gamma_1})(b_{\mu_2}^*\ot a_{\gamma_2})&-&\sum_{\beta \in J} b_{\Delta;\beta}^{\mu_1,\mu_2} (b_{\beta}^*\ot a_{\alpha}),\\
\sum_\gamma a_{1}^{\gamma} (b_{\mu}^*\ot a_{\gamma})&-&b_{1}^{\mu},\\
a_{\nu;\alpha} &-&\sum_\beta b_{\nu;\beta} (b_{\beta}^*\ot a_{\alpha}) \\
\end{eqnarray*}
for all choices of $\alpha,\alpha_1,\alpha_2\in I$ and $\mu,\mu_1,\mu_2\in J$.
It can now be shown that $\rho^{BA}:A\to B\ot \Aa_{BA}$, given by the composition
\[
\xymatrix{A \ar[r]^-{\rho_0} & B\ot (B^*\ot A) \ar@{^{(}->}[r]& B\ot T(B^*\ot A)  \ar@{->>}[r] & B\ot \Aa_{BA}}
\]
is the universal comeasuring from $A$ to $B$. By applying this coaction twice, we find that the cocomposition of the associated Hopf-opcategory is given on generators by the formula
\begin{eqnarray}
&\delta_{ABC}: \Aa_{AC}\to \Aa_{AB}\ot \Aa_{BC} \nonumber\\
&\delta_{ABC}(a^*\ot c)=\sum_{\beta\in J} (a^*\ot b_\beta) \ot (b^*_\beta\ot c).\eqlabel{cocomp}
\end{eqnarray}

We will now apply the construction to the case of Frobenius algebras of the form $k[G]$ where $G$ is a finite group. The algebra structure is in this case given by the usual group algebra and the comultiplication and counit on base elements (i.e. elements of $G$) is given by
\begin{eqnarray*}
\Delta_{G}(g)&=&\sum_{h \in G} g h^{-1} \otimes h,\\
\nu(e_G)=1,&& \nu(g)=0, g\neq e_G,
\end{eqnarray*}
where $e_G$ is the unit element of $G$. Therefore the only non-zero structure constants are $a_{m;x,y}^{xy}=1$, $a_{\Delta;xy}^{x,y}=1$, $a_{1}^{1_G}=1$, $a_{\nu,1_G}=1$ for $x,y\in G$.

Combining the above, we find that the universal comeasuring algebra $\Aa_{k[G],k[H]}$ for some finite groups $G$ and $H$ is given by the quotient of the free $k$-algebra over the set $\{ g^*\ot h~|~ g\in G,h\in H\}$, by the ideal generated by
\begin{eqnarray} \eqlabel{groupalg1}
((ab)^*\ot x)-\sum_{h\in H} (a^* \ot xh^{-1})(b^* \ot h) & a,b\in G , x\in H, \\ \eqlabel{groupalg2}
(a^*\ot xy)-\sum_{g\in G}(a^*(g^{-1})^*\ot x)(g^* \ot y) & a\in G, x,y \in H,\\ \eqlabel{groupalg3}
\left(1_{G}^{*} \otimes 1_{H}\right)-1,&  \\ \eqlabel{groupalg4}
\left(1_{G}^{*} \otimes x\right) &  1_{H}\neq x\in H, \\ \eqlabel{groupalg5}
\left(a^* \otimes 1_{H}\right) &  1_G \neq a \in G .
\end{eqnarray}
Remark that all equations factored out are symmetric in the $G$ and $H$ component. More precisely, we have the following result. 

\begin{proposition} \label{symmeas}
Let $G$ and $H$ be finite groups and $k$ a field. Then the assignment $h^*\ot g\mapsto g^*\ot h$ induces an algebra isomorphism $\mathcal{A}_{k[H],k[G]} \simeq\mathcal{A}_{k[G],k[H]}$. Therefore also $\mathcal{C}_{k[H],k[G]} \simeq \mathcal{C}_{k[H],k[G]}$ as coalgebras.
\end{proposition}

The previous proposition is in fact a reincarnation of the Hopf category isomorphism $\pi$ from Proposition~\ref{isomorphismpi}. To see this, observe that the isomorphism $\iota_{k[G]}:k[G]\to k[G]^*$, is given by the formula $\iota_{k[G]}(g)=\nu(g-)=(g^{-1})^*$. Furthermore, following the proof of Theorem \ref{univcoact}, the antipode $S_{k[G],k[H]}:\Aa_{k[G],k[H]}\to \Aa_{k[H],k[G]}$ is given by the formula
\begin{equation}\eqlabel{antipodegroup}
g^*\ot x \mapsto  \sum_{h\in H} \nu_{k[H]}(hx)h^*\ot g^{-1} = (x^{-1})^* \ot g^{-1}
\end{equation}
By combining the above formula, one then indeed recovers the commutative diagram from Theorem~\ref{factor}.

\begin{example} 
For the trivial group, the group algebra is just the base field $k$, which is the case that we already treated at the beginning of this section. The next simplest example is the cyclic group with two elements $G=\{e,g\}$. 
For any group $H$ non-isomorphic to $C_{2}$ we find $\mathcal{A}_{k[C_2],k[H]}=0$. If $H$ is the trivial group, then this follows from the reasoning in the beginning of this section. Now, let $|H|\geq 3$ and let $h$ be an arbitrary non trivial element in $H$. Then there exists an element $1_H\neq k\in H$ such that $hk\neq 1_H$. Now we can use \equref{groupalg2} and \equref{groupalg4} to show:
$$(g^*\ot h)=(g^*\ot hkk^{-1})=(g^*\ot hk )(e^*\ot k^{-1})+ (e^*\ot hk)(g^*\ot k^{-1})=(g^*\ot hk )0+ 0(g^*\ot k^{-1})=0$$
Since $h$ was arbitrary and $(e^*\ot h)=0$ by \equref{groupalg4}, the only generator of $\mathcal{A}_{k[C_2],k[H]}$ is $(e^*\ot 1_H)$. However using \equref{groupalg3} and \equref{groupalg2} we obtain
$$1=(e^*\ot 1_H)=(e^*\ot hh^{-1})=(e^*\ot h)(e^*\ot h^{-1})+(g^*\ot h)(g^*\ot h^{-1})=0$$
in $\mathcal{A}_{k[C_2],k[H]}$.
Therefore, $\mathcal{A}_{k[C_2],k[H]}=\mathcal{A}_{k[H],k[C_2]}=\mathcal{C}_{k[C_2],k[H]}=\mathcal{C}_{k[H],k[C_2]}=0$ for all groups $H \neq C_{2}$.
\end{example}

\begin{example} 
Let us now compute the universal coacting Hopf algebra of $k[C_{2}]$, which is generated as an algebra by the element $x:=\left(g^{*} \otimes g\right)$ (remark that the other initial generators of the free algebra become either $1$ or $0$ in the quotient). We can use \equref{groupalg4}, \equref{groupalg2} and \equref{groupalg3} to show:
$$x^{2}=(g^{*} \ot g)(g^{*} \ot g)=(g^{*} \ot g)(g^{*} \ot g)	+(e^{*} \ot g)(e^{*} \ot g)=(e^{*} \otimes gg)=(e^*\ot e)=1.$$
Therefore we find a representation of the universal coacting Hopf algebra 
$$\mathcal{A}_{k[C_2],k[C_{2}]}=k[x] /\left\langle x^{2}-1\right\rangle.$$ 
By \equref{cocomp}, the comultiplication is then given by
$$\delta(x)=\delta(g^*\ot g)=(g^*\ot e)\ot (e^*\ot g) + (g^*\ot g)\ot (g^*\ot g) = x\ot x$$
and of course $\delta(1)=1\ot 1$. 
Since both $x$ and $1$ are grouplike we have $\epsilon(1)=\epsilon(x)=1$. The antipode $S$ is the identity.
We conclude that the universal coacting Hopf algebra in this case is isomorphic to the group algebra $k[C_{2}]$. 

It is well-known that the dual Hopf algebra of $k[C_{2}]$ is isomorphic to $k[C_{2}]$ if the characteristic of $k$ is different from $2$
and is generated by the two grouplike elements $\left(1^{*}+x^{*}\right),\left(1^{*}-x^{*}\right)$. These grouplike elements correspond to the identity on $k[C_2]$ and the Frobenius automorphism sending $e \mapsto e$ and $g\mapsto -g$.\\
 For $\Char(k)=2$ the two grouplikes (and the corresponding two automorphisms of the Frobenius algebra) coincide. However, in this case $x^{*}$ is a primitive element: 
 $$\delta'(x^*)=(1^{*} \ot x^{*})+(x^{*} \ot 1^{*})= (1^{*} \ot x^{*})+(x^{*} \ot 1^{*})+2(x^{*} \ot x^{*})=(1^{*} + x^{*})\ot x^*+x^*\ot (1^{*} +x ^{*}).$$ 
This primitive element corresponds to a bi-derivation $D: k[C_{2}] \to k[C_{2}]$ (see beginning of Section~\ref{HopfFrob}). Therefore in the case of characteristic 2 the universal acting Hopf algebra is generated by the identity and a bi-derivation. Already in this small example we see therefore that the universal acting Hopf algebra can be different from a groupalgebra.
\end{example}

\begin{example}
Next in line is the group of order $3$, $C_{3}:=\{e,a,b\}$ with $a^2=b,ab=e$ and neutral element $e$. For the universal coacting algebra of $k[C_{3}]$ we find  by \equref{groupalg4} and \equref{groupalg5} that the only generators different from $0$ and $1$ are $a^{*} \otimes a$, $a^{*} \otimes b$, $b^{*} \otimes a$ and $b^{*} \otimes b$. 
Furthermore, for any $x\neq e\neq y$ in $C_3$ we have 
$$(a^{*} \ot x)(a^{*} \ot y)=(a^{*} \ot x)(a^{*} \ot y)+(e^{*} \ot x)(b^{*} \ot y)+(b^{*} \ot x)(e^{*} \ot y)=b^{*} \ot x y,$$
by using \equref{groupalg4} and \equref{groupalg2}. Similar calculations for $b$ yields $(b^{*} \ot x)(b^{*} \ot y)=a^*\ot xy$. By symmetry we also have $(x^{*} \ot a)(y^{*} \ot a)=(xy)^*\ot b$ and  $(x^{*} \ot b)(y^{*} \ot b)=(xy)^*\ot a$. 
Therefore, the only additional base elements for $\Aa_{k[C_3],k[C_3]}$ are
 $(a^{*} \ot a)(b^{*} \ot b)$ and $(b^{*} \ot a)(a^{*} \ot b)$.
To simplify notation we rename the elements as follows
\begin{eqnarray*}
A &:=\left(a^{*} \otimes a\right), B:=\left(b^{*} \otimes b\right), X:=\left(a^{*} \otimes a\right)\left(b^{*} \otimes b\right) \\
C &:=\left(a^{*} \otimes b\right), D:=\left(b^{*} \otimes a\right), Y:=\left(b^{*} \otimes a\right)\left(a^{*} \otimes b\right),
\end{eqnarray*}
and with this notation, the multiplication table is presented below.

\[
\begin{tabular}{|c|cccccc|}
\hline & $X$ & $ {A}$ & $ {B}$ & $ {Y}$ & $ {C}$ & $ {D}$ \\
\hline $ {X}$ & $ {X}$ & $ {A}$ & $ {B}$ & 0 & 0 & 0 \\
$ {A}$ & $ {A}$ & $ {B}$ & $ {X}$ & 0 & 0 & 0 \\
$ {B}$ & $ {B}$ & $ {X}$ & $ {A}$ & 0 & 0 & 0 \\
$ {Y}$ & 0 & 0 & 0 & $ {Y}$ & $ {C}$ & $ {D}$ \\
$ {C}$ & 0 & 0 & 0 & $ {C}$ & $ {D}$ & $ {Y}$ \\
$ {D}$ & 0 & 0 & 0 & $ {D}$ & $ {Y}$ & $ {C}$ \\
\hline
\end{tabular}
\]
One can notice that the multiplication is commutative and the unit is given by $1=X+Y (=e^*\ot e)$.
Therefore we have a six dimensional algebra presented as the direct product of two 3 dimensional algebras: $\mathcal{A}_{k[C_{3}],k[C_{3}]} \simeq k[C_{3}] \times k [C_{3}]$. However the coalgebra structure intertwines the two subalgebras. Explicitly, the comultiplication is given by the following formulas.
\begin{eqnarray*}
\delta(A)=A \otimes A+D \otimes C&& \delta(C)=C \otimes A+B \otimes C\\
\delta(B)=B \otimes B+C \otimes D&& \delta(D)=A \otimes D+D \otimes B\\
\delta(X)=\delta(A) \delta(B)=X \otimes X+Y \otimes Y&& \delta(Y)=\delta(C) \delta(D)=X \otimes Y+Y \otimes X
\end{eqnarray*}
To obtain the formula for the counit $\nu$, recall that $\nu$ is an algebra morphism satisfying $(id\ot \nu)\circ \rho=id$ from this, we can deduce that
$$\epsilon(A)=\epsilon(B)=\epsilon(X)=1 , \qquad \epsilon(C)=\epsilon(D)=\epsilon(Y)=0.$$
For the antipode $S$ we use formula \equref{antipodegroup} and find
$$S(A)=B,\quad S(B)=A,\quad S(X)=X,\quad S(Y)=Y,\quad S(C)=C,\quad S(D)=D.$$
Now we dualize the Hopf algebra to find the universal acting Hopf algebra $\mathcal{C}_{k[C_{3}],k[C_{3}]}$. It is generated by $X^{*}, A^{*}, B^{*},$ $ Y^{*}, C^{*}, D^{*}$. The coalgebra structure is given by the coalgebra structure on the Frobenius algebra $k[C_{3}]$ with generators $\{X^{*}, A^{*}, B^{*}\}$ and $\{Y^{*}, C^{*}, D^{*}\}$. 
The multiplication table comes out as:
\[
\begin{tabular}{|c|cccccc|}
\hline & $X^{*}$ & $A^{*}$ & $B^{*}$ & $Y^{*}$ & $C^{*}$ & $D^{*}$ \\
\hline$X^{*}$ & $X^{*}$ & 0 & 0 & $Y^{*}$ & 0 & 0 \\
$A^{*}$ & 0 & $A^{*}$ & 0 & 0 & 0 & $D^{*}$ \\
$B^{*}$ & 0 & 0 & $B^{*}$ & 0 & $C^{*}$ & 0 \\
$Y^{*}$ & $Y^{*}$ & 0 & 0 & $X^{*}$ & 0 & 0 \\
$C^{*}$ & 0 & $C^{*}$ & 0 & 0 & 0 & $B^{*}$ \\
$D^{*}$ & 0 & 0 & $D^{*}$ & 0 & $A^{*}$ & 0 \\
\hline
\end{tabular}
\]
The unit is given by $X^{*}+A^{*}+B^{*}$ and the antipode is given by $S^{*}(A^{*})=B^{*}$, $S^{*}(B^{*})=A^{*}$ and the identity on the other generators. The universal action $\psi: \mathcal{C}_{k[C_{3}],k[C_{3}]} \otimes k[C_{3}] \to k[C_{3}]$ is given by
\begin{eqnarray*}
\psi(X^{*}\ot e)=\psi(Y^{*}\ot e)=e \\
\psi(A^{*}\ot a)=\psi(D^{*}\ot b)=a \\
\psi(B^{*}\ot b)= \psi(C^{*}\ot a)=b 
\end{eqnarray*}
and zero everywhere else. Note that $X^{*}$ and $Y^{*}$ have the same action, but are distinguishable by the coalgebra structure. We have two grouplike elements $X^{*}+A^{*}+B^{*}$, corresponding to the identity of $k[C_3]$ and $Y^{*}+C^{*}+D^{*}$ corresponding to the automorphism of Frobenius algebras $\tau =\psi((Y^*+C^*+D^* )\ot -): k[C_3]\to k[C_3]$ given by $\tau(a)=b, \tau(b)=a$. Remark that, for example, the element $X^*+A^*+B^*$ also acts as the identity on $k[C_3]$, although the resulting element in $\mathcal{C}_{k[C_{3}],k[C_{3}]}$ is not grouplike.

If the characteristic of the field $k$ is different from $3$ and furthermore $k$ has some primitive third root of unity, then we find additional grouplike elements. Indeed, suppose $1\neq \xi \in k$ satisfies $\xi^2+\xi+1=0$, then the following elements of  $\mathcal{C}_{k[C_{3}],k[C_{3}]}$ are also grouplike: 
$$X^{*}+\xi A^{*}+\xi^{2} B^{*},\qquad X^{*}+\xi^{2} A^{*}+\xi B^{*},\quad Y^{*}+\xi C^{*}+\xi^{2} D^{*}, \quad Y^{*}+\xi^{2} C^{*}+\xi D^{*}.$$ 
As in this way, we found $6$ grouplike elements, the Hopf algebra $\mathcal{C}_{k[C_{3}],k[C_{3}]}$ has a base of grouplike elements. Since the comultiplication of $\mathcal{A}_{k[C_{3}],k[C_{3}]}$ was not cocommutative, the multiplication of $\mathcal{C}_{k[C_{3}],k[C_{3}]}$ is not commutative, hence the grouplike elements form a non-commutative group of order $6$. We can conclude that $$\mathcal{C}_{k[C_{3}],k[C_{3}]}\cong k[\mathbb S_3]$$ 
as Hopf algebra, and $\mathcal{C}_{k[C_{3}],k[C_{3}]}$ is completely determined by the set of Frobenius automorphisms of $k[C_{3}]$.

As in the example of the universal coaction of $k[C_2]$ above, we find biderivations if the field is of characteristic 3. For example, we have in this case:
\begin{eqnarray*}
& & (X^{*}+A^{*} + B^{*})  \ot(A^{*}-B^{*}) +(A^{*}-B^{*}) \ot (X^{*}+A^{*}+B^{*}) \\
&=& (X^{*} \ot A^{*}) + (A^{*} \ot A^{*}) +(B^{*} \ot A^{*})-(X^{*} \ot B^{*})-(A^{*} \ot B^{*}) -(B^{*} \ot B^{*}) \\
&+& (A^{*} \ot X^{*}) + (A^{*} \ot A^{*}) +(A^{*} \ot B^{*})-(B^{*} \ot X^{*})-(B^{*} \ot A^{*}) -(B^{*} \ot B^{*}) \\
&+& (B^{*} \ot B^{*}) + (A^{*} \ot A^{*}) - (A^{*} \ot A^{*})-(B^{*} \ot B^{*})\\ 
&=& \delta(A^{*}) + 0 -\delta(B^{*}) -0=\delta(A^{*}-B^{*})
\end{eqnarray*}
Therefore $A^*-B^*$ is a primitive element in $\mathcal{C}_{k[C_{3}],k[C_{3}]}$ and hence a biderivation.
Similarly, $\left(C^{*}-D^{*}\right)$ is a biderivation too and further any linear combination of these. In particular, in this case, 
$\mathcal{C}_{k[C_{3}],k[C_{3}]}$ is not a group algebra.
\end{example}

\begin{example}
Let $H$ be a group non isomorphic to $C_3$. Then we have that 
$$\mathcal{A}_{k[C_3],k[H]}=\mathcal{A}_{k[H],k[C_3]}=\mathcal{C}_{k[C_3],k[H]}=\mathcal{C}_{k[H],k[C_3]}=0.$$
Indeed, we use notation as above for the elements of $C_3$ and for the universal comeasuring algebra.
For any $x\neq 1_H \neq y$ in $H$ we then have the following identity in $\mathcal{A}_{k[C_3],k[H]}$:
\begin{eqnarray} \eqlabel{up}
(a^*\ot xy)=(a^*\ot x)(e^*\ot y)+(e^*\ot x)(a^*\ot y)+(b^*\ot x)(b^*\ot y)=(b^*\ot x)(b^*\ot y).
\end{eqnarray}
Similarly 
\begin{eqnarray} \eqlabel{down}
(b^*\ot xy)=(a^*\ot x)(a^*\ot y).
\end{eqnarray}
Suppose now $|H|\geq 4$ and let $z\in H$. We can write $z=xhh^{-1}y$ such that $x,h,y$ are all non trivial as well as $xh $ and $h^{-1}y$. Then we find
\begin{eqnarray*}
(b^*\ot z)&=&(a^*\ot x)(a^*\ot y)=(b^*\ot xh)(b^*\ot h^{-1})(b^*\ot h)(b^*\ot h^{-1}y)\\
&=&(b^*\ot xh)(a^*\ot h^{-1} h)(b^*\ot h^{-1}y)\\
&=&(b^*\ot xh)(a^*\ot e_H)(b^*\ot h^{-1}y)=0
\end{eqnarray*} 
using \equref{up}, \equref{down} and  \equref{groupalg5}. Since $z$ was arbitrary we obtain $(b^*\ot z)=0$ for all $z\in H$ and therefore $(a^*\ot z)=(b^*\ot x)(b^*\ot y)=0$. Hence $\Aa_{k[C_3],k[H]}$ is only generated by $(e^*\ot 1_H)$. But since $(e^*\ot 1_H)=\sum_{x\in H} (a^*\ot x^{-1})(b^*\ot x)=0$ we have that $\Aa_{k[C_3],k[H]}=0$. The other statements follow from Proposition~\ref{symmeas} and Theorem~\ref{meascomeas}.
\end{example}

The next groups to handle are those of order $4$, the first order for which we have two non-isomorphic groups: $C_4$ and $C_2\times C_2$. Although this example becomes already very large to compute completely by hand, let us make some interesting observations in this case. 
Recall that, given a field $k$ with $\Char(k)\neq 2$ and with a primitive fourth root of unity $i$, the group algebras of $C_4$ and $C_2\times C_2$ are isomorphic as Frobenius algebras. Denote $C_4:=\{1,x,x^2,x^3\}$ and $C_{2}:=\{e,g\}$ then one can check that the linear map $\phi: k[C_2\times C_2]\to k[C_4]$ given by 
\begin{eqnarray*}
(e,e)&\mapsto& 1,\\
(g,e)&\mapsto& \frac{1+i}{2}x+ \frac{1-i}{2}x^3,\\
(e,g)&\mapsto& \frac{1-i}{2}x+ \frac{1+i}{2}x^3,\\
(g,g)&\mapsto& x^2
\end{eqnarray*}
is a morphism of algebras preserving the Frobenius structure, hence is an isomorphism. Therefore, it follows that the universal measuring coalgebras $\mathcal C_{k[C_2\times C_2],k[C_4]}$ and $\mathcal C_{k[C_4],k[C_2\times C_2]}$ as well as the universal comeasuring algebras $\mathcal A_{k[C_2\times C_2],k[C_4]}$ and $\mathcal A_{k[C_4],k[C_2\times C_2]}$ are non-trivial, since they should contain at least the element associated to the above isomorphism. If now $k$ is a field of characteristic different from $2$, but not containing a primitive fourth root of unity (e.g. $k=\QQ$), then we know that $k[C_2\times C_2]$ and $k[C_4]$ are non-isomorphic, however also in this case we can conclude that the universal measuring coalgebras $\mathcal C_{k[C_2\times C_2],k[C_4]}$ and $\mathcal C_{k[C_4],k[C_2\times C_2]}$ as well as the universal comeasuring algebras $\mathcal A_{k[C_2\times C_2],k[C_4]}$ and $\mathcal A_{k[C_4],k[C_2\times C_2]}$ are still non-trivial. To see this, let us first state and prove the following useful observation, showing that the universal measuring coalgebra already ``detects'' (iso)morphisms that arise only after base extension.
\begin{proposition}\label{pr:descent}
Let $A$ and $B$ be two Frobenius algebras over a field $k$. Let $\ell$ be any extension of $k$. Consider the Frobenius algebras  over $\ell$ obtained by extension of scalars: $\ell\ot_k A$ and $\ell\ot_k B$.
Then the following assertions hold.
\begin{enumerate}[(i)]
\item $\Aa_{\ell\ot_kA,\ell\ot_kB}\cong \ell\ot_k\Aa_{A,B}$
\item The dimension of $\Aa_{A,B}$ is at least the cardinality of the set of morphisms of Frobenius algebras between $\ell\ot_kB$ and $\ell\ot_kA$.
\end{enumerate}
The same holds for arbitrary $\Omega$-algebras.
\end{proposition}

\begin{proof}
\ul{(i)}.
First remark that if we fix a $k$-base $\{a_\alpha,\alpha\in I\}$ for $A$, then $\{1_\ell\ot a_\alpha, \alpha\in I\}$ is an $\ell$-base for $\ell\ot_k A$. Moreover the structure constants (belonging to $k$) for $A$ with respect to the base $\{a_\alpha,\alpha\in I\}$ are also the structure constants for $\ell\ot_k A$ with respect to the base $\{1_\ell\ot a_\alpha, \alpha\in I\}$\footnote{Hence the structure constants for $\ell\ot_k A$ with respect to the base $\{1_\ell\ot a_\alpha, \alpha\in I\}$ belong to $k$. In fact, by usual descent theory, an $\ell$-algebra is obtained by extension of scalars from a $k$-algebra exactly if there exists an $\ell$-base whose structure constants belong to $k$.}. 

Since the construction of the universal comeasuring algebra as described in the beginning of this section is purely given by generators and relations, where the former are constructed out of bases of the algebras and the latter depend on the structure constants for these bases. As explained above, bases and structure constants are preserved under base extension, hence the comeasuring algebra of base-extended algebras is the base extension of the comeasuring algebra of the initial algebras.

\ul{(ii)}. 
It is well-known that grouplike elements in a coalgebra are linearly independent and by Lemma~\ref{le:meascomposition} grouplike elements in the measuring coalgebra $\Cc_{AB}$ correspond to morphisms of Frobenius algebras from $B$ to $A$. Therefore the $\ell$-dimension of $\Cc_{\ell\ot_k A,\ell\ot_k B}$ is at least the cardinality of the set of morphisms of Frobenius algebras between $\ell\ot_kB$ and $\ell\ot_kA$. Furthermore, since $\Cc_{\ell\ot_k A,\ell\ot_k B}=\Aa^{\circ}_{\ell\ot_k A,\ell\ot_k B}$, also the $\ell$-dimension of $\Aa_{\ell\ot_k A,\ell\ot_k B}$ is at least this cardinality. Finally, by part (i), we find that the $\ell$-dimension of $\Aa_{\ell\ot_k A,\ell\ot_k B}$ equals the $k$-dimension of $\Aa_{AB}$.

For the last statement, it suffices to observe that the above reasoning not particular for Frobenius algebras, but also holds for arbitrary $\Omega$-algebras.
\end{proof}

Let us now prove a general triviality result for universal comeasuring algebras between group algebras.

\begin{proposition}\label{pr:groupgroup}
If $G$ and $H$ are finite groups and $k$ is a field such that $\Char k$ does not divide $|G|-|H|$, 
$$\mathcal{A}_{k[G],k[H]}=\mathcal{A}_{k[H],k[G]}=\mathcal{C}_{k[G],k[H]}=\mathcal{C}_{k[H],k[G]}=0.$$
\end{proposition}
\begin{proof}
As before, it suffices to prove that $\mathcal{A}_{k[G],k[H]}=0$. With notation as above we can compute
\begin{eqnarray*}
|G| 1 &=\sum_{g \in G}\left(1^*_{G} \otimes 1_{H}\right)=\sum_{g \in G}\left((g^*)^{-1} g^* \otimes 1_{H}\right)=\sum_{g \in G} \sum_{h \in H}\left((g^*)^{-1} \otimes h^{-1}\right)(g^* \otimes h) \\
&=\sum_{h \in H}\left(1^*_{G} \otimes h^{-1} h\right)=\sum_{h \in H}\left(1_{G}^* \otimes 1_{H}\right)=|H| 1
\end{eqnarray*}
If $\Char k$ does not divide $|G|-|H|$, then we can conclude that $0=1$ in $\mathcal{A}_{k[G],k[H]}=0$.
\end{proof}

We will now look into another family of Frobenius algebras, namely matrix algebras.

\begin{example}
Fix $n\in\mathbb N_0$ and consider the algebra $\mathcal{M}_{n}(k)$ of $n\times n$ matrix algebras with entries in $k$. We denote the canonical basis elements $\left\{E^n_{i j}\right\}_{i, j=1}^{n}$, which are zero everywhere except at the $(i, j)$-th position, where we have $1$. The Frobenius structure is given by 
\begin{eqnarray*}
\mu(E^n_{i j} \ot E^n_{k l})=\delta_{j k} E^n_{i l} & I_{n}=\sum_{i} E^n_{i i}\\
\Delta(E^n_{i j})=\sum_{k} E^n_{i k} \ot E^n_{k i} & \nu(E^n_{i j})=\delta_{i j},
\end{eqnarray*}
where $\delta_{ij}$ is the Kronecker-delta. 
Notice that just as group algebras, matrix algebras are special Frobenius since $\mu \circ \Delta=n \cdot \Id$ and therefore the two behave fairly similar. The universal comeasuring algebra is now given by $\mathcal{A}_{\Mm_m,\Mm_n}=T\left(\mathcal{M}_{m}^{*} \otimes \mathcal{M}_{n}\right) / I$ factoring out the relations
\begin{eqnarray} 
\eqlabel{matrix1}
\delta_{yz}E_{uv}^{m*} \ot E_{i\ell}^n &=& \sum_{p=1}^n (E_{uy}^{m*} \ot E^n_{ip})(E_{zv}^{m*}\ot E^n_{p\ell})  \\
\eqlabel{matrix2}
\delta_{jk}E_{uv}^{m*} \ot E_{i\ell}^n &=& \sum_{w=1}^m (E_{uw}^{m*} \ot E^n_{ij})(E_{wv}^{m*}\ot E^n_{k\ell})  \\
\eqlabel{matrix3} \delta_{i\ell} &=& \sum_{w=1}^m E_{ww}^{m*}\ot E_{i\ell} \\
\eqlabel{matrix4} \delta_{uv} &=& \sum_{p=1}^n E_{uv}^{m*}\ot E_{pp}^n  
\end{eqnarray}
for all $i,j,k,\ell=1,\ldots,n$ et $u,v,y,z=1,\ldots,m$.
\end{example}

Similarly to group algebras we find that the existence of non-zero measurings between matrix algebras puts severe restriction on their dimensions.
\begin{proposition}\label{pr:matrixmatrix}
Let $k$ be a field and $n,m\in\mathbb N_0$ such that $\Char k$ does not divide $n-m$. Then 
$$\mathcal{A}_{\mathcal{M}_{n}(k),\mathcal{M}_{m}(k)}=\mathcal{A}_{\mathcal{M}_{m}(k),\mathcal{M}_{n}(k)}=\mathcal{C}_{\mathcal{M}_{n}(k),\mathcal{M}_{m}(k)}=\mathcal{C}_{\mathcal{M}_{m}(k),\mathcal{M}_{n}(k)}=0.$$
\end{proposition}
\begin{proof}
 Denote the canonical basis by $E^m_{k l}$ and $E^n_{i j}$. Then
$$m1=\sum_{k}^{m} \delta_{k k}
=\sum_{k}^{m} \sum_{i}^{n} ((E_{k k}^{m})^*\ot E^n_{i i}) 
=\sum_{i}^{n} \delta_{i i}=n1$$
by using \equref{matrix3} and \equref{matrix4} and the definitions of the units in $\Mm_n$ and $\Mm_m$. Therefore $1=0$ in $\mathcal{A}_{\mathcal{M}_{m}(k),\mathcal{M}_{n}(k)}$ if $\Char k$ does not divide $n-m$.
\end{proof}

To finish this section we consider measurings between Frobenius group algebras and matrix algebras.

\begin{proposition} \label{pr:groupmatrix}
Let $k$ be a field, $G$ a finite group and $n\in\mathbb N_0$ such that $\Char k$ does not divide $n-|G|$ or $n-1$. Then 
$$\mathcal{A}_{\mathcal{M}_{n}(k),k[G]}=\mathcal{A}_{k[G],\mathcal{M}_{n}(k)}=\mathcal{C}_{\mathcal{M}_{n}(k),k[G]}=\mathcal{C}_{k[G],\mathcal{M}_{n}(k)}=0.$$
\end{proposition}
\begin{proof} 
The universal comeasuring algebra is now given by $\mathcal{A}_{k[G],\Mm_n(k)}=T\left(k[G]^{*} \otimes \mathcal{M}_{n}\right) / I$ factoring out the relations
\begin{eqnarray} 
\eqlabel{matrixgroup1}
((ab)^*\ot E^n_{ij})&=&\sum_{k=1}^n (a^* \ot E^n_{ik})(b^* \ot E^n_{kj})  \\ 
\eqlabel{matrixgroup2}
\delta_{k\ell}(a^*\ot E^n_{ij})&=&\sum_{g\in G} ((a g^{-1})^* \ot E^n_{ik})(g^*\ot E^n_{\ell j})\\  
\eqlabel{matrixgroup3}
\delta_{ij}&=&1_{G}^{*} \otimes E_{ij} \\  
\eqlabel{matrixgroup4}
1&=&\sum_{k=1}^n(1_{G}^* \otimes E_{kk}^n)	 
\end{eqnarray}
for all $ a,b\in G$ and $i,j,k,\ell=1,\ldots,n$.

Using relations \equref{matrixgroup4} and \equref{matrixgroup3} we find:
$$1=\sum_{k=1}^{n} (e_{G}^{*} \otimes E_{kk})=\sum_{k=1}^{n} \delta_{kk}=n1.$$
So $1=0$ whenever $\Char k$ does not divide $n-1$. 

On the other hand, using \equref{matrixgroup1}, \equref{matrixgroup2} and \equref{matrixgroup4} we find
\begin{eqnarray*}
|G|1&=&\sum_{g \in G} 1 =\sum_{g \in G}\sum_{i=1}^n ((gg^{-1})^* \ot E^{n}_{ii})\\
&=&\sum_{g \in G}\sum_{i=1}^n \sum_{j=1}^n (g^* \otimes E^{n}_{ij})((g^*)^{-1} \ot E^n_{ji})  \\
&=& \sum_{i=1}^n \sum_{j=1}^n (1^*_G \otimes E^{n}_{ii}) =\sum_{j=1}^n 1=n1
\end{eqnarray*}
hence $1=0$ whenever $\Char k$ does not divide $n-|G|$.
\end{proof}

\section{Conclusions and outlook}

\subsection{Remaining questions}

In this paper, we studied measuring coalgebras and comeasuring algebras between Frobenius algebras. As a main result we showed that the semi-Hopf category (i.e.\ the coalgebra enriched category) of universal measuring coalgebras between Frobenius algebras admits an invertible antipode. We discussed duality between measuring and comeasurings and computed several explicit examples. Based on these results and examples, we formulate a few remaining questions that require further investigations. 

\begin{enumerate}
\item Recall that if $\ell$ is a field extension of $k$, then a $k$-algebra $A$ is called a {\em form} of an $\ell$-algebra $R$, if and only if $R\cong \ell\ot_k A$. So far, the only examples we have where $\Cc_{A,B}$ is non-zero for some Frobenius $k$-algebras $A$ and $B$, is when $A$ and $B$ are forms of the same $\ell$-algebra for some field extension $\ell$ of $k$. Therefore, and in view of Proposition~\ref{pr:descent} above, we pose the following question:
\begin{quote}
Is it true that $\Cc_{A,B}$ is nonzero for some some Frobenius $k$-algebras $A$ and $B$ if and only if $A$ and $B$ are forms of the same $\ell$-algebra for some field extension $\ell$ of $k$ ?
\end{quote}
For the above to hold, it would be sufficient to see that there exists a base extension $\ell$ of $k$, such that $\ell\ot_k\Cc_{A,B}$ has a grouplike element, since such a grouplike element acts as a morphism of Frobenius algebras, which is therefore an isomorphism.
The results we have proven so far (see Propositions~\ref{pr:groupgroup}, \ref{pr:groupmatrix} and \ref{pr:matrixmatrix}), already show that there are strong dimension restrictions for group algebras and matrix algebras in order to have existence of non-zero (co)measurings between them. Since both group algebras and matrix algebras are so-called {\em special} Frobenius, meaning that the composition of multiplication with comultiplicaton equals a scalar multiple of the identity, one could wonder whether it is possible to generalize the triviality conditions to this setting. 
\item In all examples we computed, the universal measuring coalgebras and comeasuring algebras were finite dimensional. However, there is a priori no reason why this should always be the case. We therefore ask
\begin{quote}
Do there exist Frobenius algebras $A$ and $B$, whose universal measuring coalgebra $\Cc_{AB}$ is infinite dimensional ?
\end{quote}
Furthermore, one could also wonder, if there is a formula to compute the dimension of the universal measuring coalgebra. Based on the very small examples we have, one already sees that the dimension of the universal acting Hopf algebra on a Frobenius algebra grows more rapidly than the dimension of the Frobenius algebra itself (the dimension of $\Cc_{k[C_2],k[C_2]}$ being $2$ and the dimension of $\Cc_{k[C_3],k[C_3]}$ being $6$, computer algebra computations showed that the dimension of $\Cc_{\CC[C_4],\CC[C_4]}$ is at least 96).
\item One could wonder, if there exist other types of $\Omega$-algebras such that the associated measuring semi-Hopf category is Hopf ? Or even stronger: what are the conditions on $\Omega$ and its $\Omega$-algebras for this to be the case ? Based on \seref{duality2} we expect that this could be the case for $\Omega$'s that are ``self-dual''. 
\end{enumerate}

\subsection{Motivation from topological quantum field theory}

Finally, let us mention some possible application of our work to Topological Quantum Field Theory which was in fact the initial motivation for our investigations. Recall that a Topological Quantum Field Theory is nothing else than a symmetric monoidal functor $Z$ from a suitably defined category of $n$-cobordisms to the category of vector spaces. One way to define the needed category of cobordisms is as follows (many variations are possible). Objects in such a category are $n-1$-dimensional oriented manifolds (with the possibility to restrict for example to closed ones or open ones). A morphism between two such objects $M$ and $N$, is a diffeormorphism class $\Sigma$ of $n$-dimensional oriented manifolds, together with diffeomorphisms of oriented manifolds between $M$ and the in-boundary of $\Sigma$, and between $N$ and the out-boundary of $\Sigma$. Composition is obtained from gluing, tensor product is just disjoint union.

In the one-dimensional case, where the cobordisms consist of {\em strings} (curved lines) between (ordered, finite) families of (oriented) points, a TQFT is determined by a finite dimensional vector space $V$. Indeed, possible cobordisms, generating the category are pictured below. The functor $Z$ is then completely determined by the image of the positive point, which is the vector space $V$, and sends the negative point to its dual vector space. The cup and cap strings are mapped to evaluation and coevaluation morphisms. 
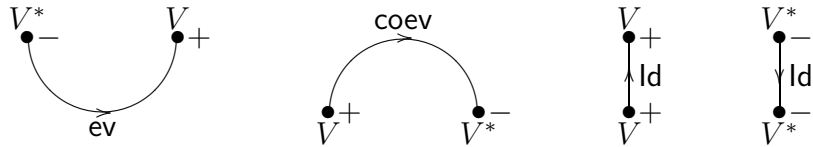
\[
\xy
(-10,0)*{\bullet}="Vd",*+!D{V^*},*+!L{-}; (10,0)*{\bullet}="V",*+!D{V},*+!L{+};
(0,0)*\cir<1cm>{d^u};  (0,-10)*+!U{\ev};
(-10,-10);(1,-10) **\dir{} ?>* \dir{>};
(30,-10)*{\bullet}="V",*+!U{V},*+!L{+}; (50,-10)*{\bullet}="Vd",*+!U{V^*},*+!L{-};
(40,-10)*\cir<1cm>{u^d};  (40,0)*+!D{{\sf coev}};
(30,0);(41,0) **\dir{} ?>* \dir{>};
(70,0)*{\bullet}="V1",*+!D{V},*+!L{+}; (70,-10)*{\bullet}="V2",*+!U{V},*+!L{+}; (70,0) ; (70,-10) **\dir{-} ;  (70,-5),*+!L{{\sf Id}};
(70,-10);(70,-4) **\dir{} ?>* \dir{>};
(90,0)*{\bullet}="V1",*+!D{V^*},*+!L{-}; (90,-10)*{\bullet}="V2",*+!U{V^*},*+!L{-}; (90,0) ; (90,-10) **\dir{-} ;  (90,-5),*+!L{{\sf Id}};
(90,0);(90,-6) **\dir{} ?>* \dir{>};
\endxy
\]
The ``dual base property'' then is a consequence of the following identities in the category of cobordisms.
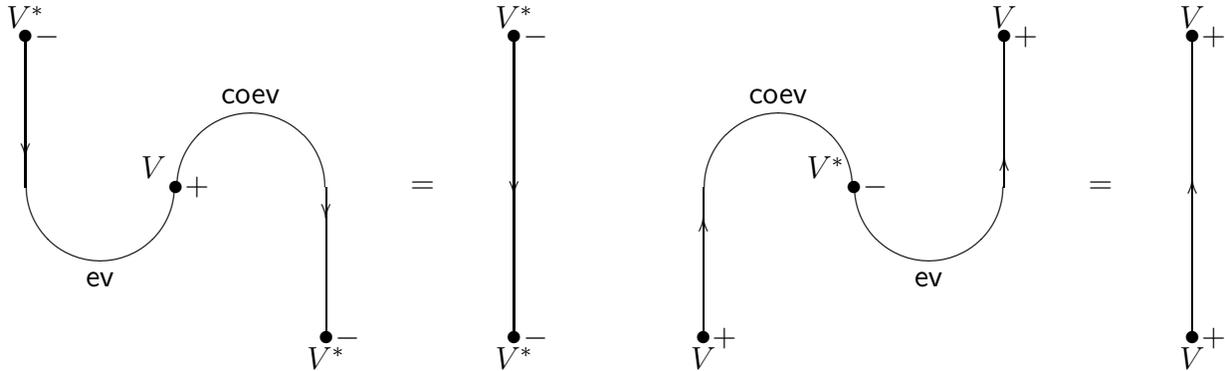
\[
\xy
(0,0)*{\bullet},*+!D{V^*},*+!L{-};
(0,0) ; (0,-20) **\dir{-};
(0,0);(0,-16) **\dir{} ?>* \dir{>};
(10,-20)*\cir<1cm>{d^u};  (10,-30)*+!U{\ev};
(20,-20)*{\bullet},*+!DR{V},*+!L{+};
(30,-20)*\cir<1cm>{u^d};  (30,-10)*+!D{{\sf coev}};
(40,-20) ; (40,-40) **\dir{-};
(40,0);(40,-24) **\dir{} ?>* \dir{>};
(40,-40)*{\bullet},*+!U{V^*},*+!L{-};
(50,-20)*+!L{=};
(65,0)*{\bullet},*+!D{V^*},*+!L{-}; (65,-40)*{\bullet},*+!U{V^*},*+!L{-};
(65,0) ; (65,-40) **\dir{-};
(65,0);(65,-21) **\dir{} ?>* \dir{>};
\endxy \qquad\qquad
\xy 
(40,0)*{\bullet},*+!D{V},*+!L{+};
(40,0) ; (40,-20) **\dir{-};
(40,-40);(40,-16) **\dir{} ?>* \dir{>};
(30,-20)*\cir<1cm>{d^u};  (30,-30)*+!U{\ev};
(20,-20)*{\bullet},*+!DR{V^*},*+!L{-};
(10,-20)*\cir<1cm>{u^d};  (10,-10)*+!D{{\sf coev}};
(0,-20) ; (0,-40) **\dir{-};
(0,-40);(0,-24) **\dir{} ?>* \dir{>};
(0,-40)*{\bullet},*+!U{V},*+!L{+};
(50,-20)*+!L{=};
(65,0)*{\bullet},*+!D{V},*+!L{+}; (65,-40)*{\bullet},*+!U{V},*+!L{+};
(65,0) ; (65,-40) **\dir{-};
(65,-40);(65,-19) **\dir{} ?>* \dir{>};
\endxy
\]

In the two-dimensional {\em open} case, cobordisms are strips between (ordered, finite) families of (oriented) line segments and a TQFT in this case is completely characterized by a symmetric Frobenius algebra (recall that {\em symmetric} means that the multiplication composed with counit is invariant under twisting of the arguments), see \cite[Corollary 4.5]{LauPfe}. The possible cobordisms generating the category are again pictured below. 
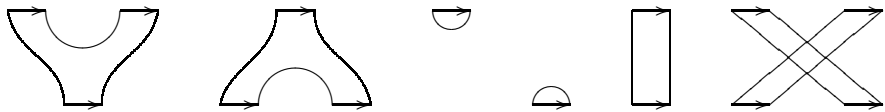
\[
\xy
(0,0);(7.5,-12.5) **\crv{(1,-6)&(7.5,-6.5)};
(10,0)*\cir<0.5cm>{d^u};  
(20,0) ; (12.5,-12.5) **\crv{(19,-6)&(12.5,-6.5)};
(0,0);(5,0) **\dir{-} ?>* \dir{>};
(15,0);(20,0) **\dir{-} ?>* \dir{>};
(7.5,-12.5);(12.5,-12.5) **\dir{-} ?>* \dir{>};
\endxy \qquad
\xy
(0,-12.5);(7.5,0) **\crv{(1,-7.5)&(7.5,-6)};
(10,-12.5)*\cir<0.5cm>{u^d};  
(20,-12.5) ; (12.5,0) **\crv{(19,-6.5)&(12.5,-6)};
(0,-12.5);(5,-12.5) **\dir{-} ?>* \dir{>};
(15,-12.5);(20,-12.5) **\dir{-} ?>* \dir{>};
(7.5,0);(12.5,0) **\dir{-} ?>* \dir{>};
\endxy \qquad
\xy
(0,0); (5,0) **\dir{-} ?>* \dir{>};
(2.5,0)*\cir<0.25cm>{d^u}; 
\endxy \qquad
\xy
(0,-12.5); (5,-12.5) **\dir{-} ?>* \dir{>};
(2.5,-12.5)*\cir<0.25cm>{u^d}; 
\endxy \qquad
\xy
(0,-12.5); (5,-12.5) **\dir{-} ?>* \dir{>};
(0,0); (5,0) **\dir{-} ?>* \dir{>};
(0,-12.5); (0,0) **\dir{-};
(5,0);(5,-12.5) **\dir{-};
\endxy\qquad
\xy
(0,0); (5,0) **\dir{-} ?>* \dir{>};
(15,0); (20,0) **\dir{-} ?>* \dir{>};
(0,-12.5); (5,-12.5) **\dir{-} ?>* \dir{>};
(15,-12.5) ; (20,-12.5) **\dir{-} ?>* \dir{>};
(15,-12.5); (0,0) **\dir{-};
(5,0);(20,-12.5) **\dir{-};
(15,0); (0,-12.5) **\dir{-};
(5,-12.5);(20,0) **\dir{-};
\endxy
\]
The following identity in the category of cobordisms visualizes the Frobenius property, the functor $Z$ now sends the oriented line segment to a Frobenius algebra $A$, and it sends the line segement with reversed orientation to the dual Frobenius algebra $A^*$.
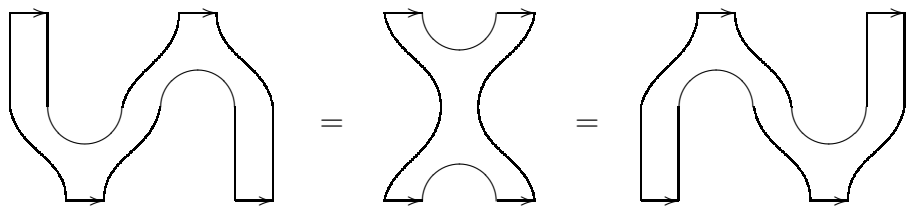
\[
\xy
(0,-12.5);(7.5,-25) **\crv{(1,-18.5)&(7.5,-19)};
(10,-12.5)*\cir<0.5cm>{d^u};  
(20,-12.5) ; (12.5,-25) **\crv{(19,-18.5)&(12.5,-19)};
(30,-25);(35,-25) **\dir{-} ?>* \dir{>};
(15,-12.5);(22.5,0) **\crv{(16,-7.5)&(22.5,-6)};
(25,-12.5)*\cir<0.5cm>{u^d};  
(35,-12.5) ; (27.5,0) **\crv{(34,-6.5)&(27.5,-6)};
(7.5,-25);(12.5,-25) **\dir{-} ?>* \dir{>};
(22.5,0);(27.5,0) **\dir{-} ?>* \dir{>};
(0,-12.5); (0,0) **\dir{-} ;
(5,0);(5,-12.5) **\dir{-} ;
(0,0);(5,0) **\dir{-} ?>* \dir{>};
(30,-12.5); (30,-25) **\dir{-};
(35,-25);(35,-12.5) **\dir{-};
(40,-15),*+!L{=};
\endxy\quad
\xy
(0,0);(7.5,-12.5) **\crv{(1,-6)&(7.5,-6.5)};
(10,0)*\cir<0.5cm>{d^u};  
(20,0) ; (12.5,-12.5) **\crv{(19,-6)&(12.5,-6.5)};
(0,0);(5,0) **\dir{-} ?>* \dir{>};
(15,0);(20,0) **\dir{-} ?>* \dir{>};
(0,-25);(7.5,-12.5) **\crv{(1,-20)&(7.5,-18.5)};
(10,-25)*\cir<0.5cm>{u^d};  
(20,-25) ; (12.5,-12.5) **\crv{(19,-19)&(12.5,-18.5)};
(0,-25);(5,-25) **\dir{-} ?>* \dir{>};
(15,-25);(20,-25) **\dir{-} ?>* \dir{>};
\endxy\quad
\xy
(-10,-15),*+!L{=};
(15,-12.5);(22.5,-25) **\crv{(16,-18.5)&(22.5,-19)};
(25,-12.5)*\cir<0.5cm>{d^u};  
(35,-12.5) ; (27.5,-25) **\crv{(34,-18.5)&(27.5,-19)};
(30,0);(35,0) **\dir{-} ?>* \dir{>};
(22.5,-25);(27.5,-25) **\dir{-} ?>* \dir{>} ;
(0,-12.5);(7.5,0) **\crv{(1,-7.5)&(7.5,-6)};
(10,-12.5)*\cir<0.5cm>{u^d};  
(20,-12.5) ; (12.5,0) **\crv{(19,-6.5)&(12.5,-6)};
(0,-25);(5,-25) **\dir{-} ?>* \dir{>};
(7.5,0);(12.5,0) **\dir{-} ?>* \dir{>};
(0,-12.5); (0,-25) **\dir{-};
(5,-25);(5,-12.5) **\dir{-};
(30,-12.5); (30,0) **\dir{-};
(35,0);(35,-12.5) **\dir{-};
\endxy
\]

Now, it is well-known that the endomorphisms of a finite dimensional vector space form a symmetric Frobenius algebra (isomorphic to a matrix algebra). Hence, it is possible to construct a $2$-dimensional (open) TQFT out of a $1$-dimensional one by moving from the object characterizing the $1$-dimensional TQFT (i.e. the finite dimensional vector space) to its endomorphism object (i.e. the matrix algebra). Geometrically, this can be interpreted as follows. If we consider a $2$-cobordism, that is, a strip between dotted line segments as in the example below, and we only look to the solid boundaries of this strip, those can be viewed as $1$-cobordisms between the endpoints of the line segments, to which we can apply a $1$-dimensional TQFT. In the example below, we see then indeed that the multiplication of the algebra associated to the $2$-dimensional TQFT constructed this way, arises by tensoring the evaluation map on both sides with identities, which is exactly the multiplication of the endomorphism algebra $\End(V)\cong V\ot V^*$ of the vector space $V$ associated to the $1$-dimenional TQFT. 
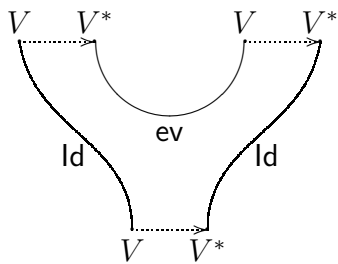
\[
\xy
(0,0);(15,-25) **[red]\crv{(2,-12)&(15,-13)};
(30,-15)*[red]+!L{{\sf Id}};
(20,0)*[red]\cir<1cm>{d^u};  
(20,-10)*[red]+!U{\ev};
(40,0) ; (25,-25) **[red]\crv{(38,-12)&(25,-13)};
(10,-15)*[red]+!R{{\sf Id}};
(0,0);(10,0) **\dir{.} ?>* \dir{>};
(30,0);(40,0) **\dir{.} ?>* \dir{>};
(15,-25);(25,-25) **\dir{.} ?>* \dir{>};
(0,0)*{\cdot},*+!D{V},*+!L{};
(10,0)*{\cdot},*+!D{V^*},*+!L{};
(30,0)*{\cdot},*+!D{V},*+!L{};
(40,0)*{\cdot},*+!D{V^*},*+!L{};
(15,-25)*{\cdot},*+!U{V},*+!L{};
(25,-25)*{\cdot},*+!U{V^*},*+!L{};
\endxy 
\]

The question which triggered us to start this work was if it would be possible to make a similar construction as above one dimension higher. That is, whether the ``endomorphisms of a $2$-dimensional TQFT'' could give rise to a $3$-dimensional one. Let us try to make this a bit more precise. When we say ``endomorphisms of a $2$-dimensional TQFT'', we in fact mean the ``endomorphism object'' of the algebraic structure associated to the $2$-dimensional TQFT, which is the (symmetric) Frobenius algebra. If we interpret the previously mentioned ``endomorphism object'' as the universal measuring coalgebra (which is a bialgebra by Proposition~\ref{prop:sHcatOmega}) from the considered Frobenius algebra to itself, then the results of our paper (see Theorem~\ref{univact}) show that this is a Hopf algebra. Furthermore it is known (see \cite{DSS}) that Hopf algebras, or more precisely their representation categories, give indeed rise to $3$-dimensional TQFTs. So far however, we did not obtain a (satisfying) geometric interpretation of this observation, but we hope that future investigations will lead to such.

\end{document}